\documentclass[12pt]{amsart}
\usepackage{latexsym,amssymb,amsmath, amsfonts}
\usepackage{graphicx}
\usepackage{pifont}
\usepackage{mathrsfs}
\usepackage{epsfig,verbatim}
\usepackage{xcolor}

\numberwithin{equation}{section}

\newtheorem{theorem}{Theorem}[section]
\newtheorem{lemma}[theorem]{Lemma}

\newtheorem{Definition}{Definition}[section]
\newenvironment{definition}{\begin{Definition}\rm}{\end{Definition}}
\newtheorem{Rem}{Remark}[section]
\newenvironment{remark}{\begin{Rem}\rm}{\end{Rem}}

\newcommand\R{\mathbb{R}}

\newcommand\be{\begin{equation}}
\newcommand\ee{\end{equation}}
\newcommand\bea{\begin{eqnarray}}
\newcommand\eea{\end{eqnarray}}
\newcommand\beaa{\begin{eqnarray*}}
\newcommand\eeaa{\end{eqnarray*}}

\newcommand\bR{\mathbb{R}}

\newcommand{\ep}{\varepsilon}

\newcommand\tu{\tilde{u}}
\newcommand\tv{\tilde{v}}

\newcommand\bNu{{\mathcal{N}_1[u]}}
\newcommand\bNv{{\mathcal{N}_2[v]}}
\newcommand\upsi{{\underline{\psi}}}
\newcommand\uphi{{\underline{\phi}}}
\newcommand\ophi{{\overline{\phi}}}
\newcommand\opsi{{\overline{\psi}}}
\newcommand\phip{{\phi^+}}
\newcommand\phim{{\phi^-}}
\newcommand\psip{{\psi^+}}
\newcommand\psim{{\psi^-}}

\newcommand\uns{\underline{s}}
\newcommand\ld{{\lambda}}
\newcommand\expl{e^{\lambda_1 z}}
\newcommand\tl{\tilde{\lambda}}

\begin{document}

\title[nonlocal dispersal and climate change]
{Persistence of species in a predator-prey system with climate change and either nonlocal or local dispersal}

\author[W. Choi]{Wonhyung Choi}
\address{Department of Mathematics, Tamkang University, Tamsui, New Taipei City 251301, Taiwan}
\email{whchoi@korea.ac.kr}

\author[T. Giletti]{Thomas Giletti}
\address{Institut Elie Cartan Lorraine, UMR 7502, University of Lorraine, 54506 Vandoeuvre-l\`es-Nancy, France}
\email{thomas.giletti@univ-lorraine.fr}

\author[J.-S. Guo]{Jong-Shenq Guo}
\address{Department of Mathematics, Tamkang University, Tamsui, New Taipei City 251301, Taiwan}
\email{jsguo@mail.tku.edu.tw}

\thanks{Date: \today. Corresponding author: J.-S Guo}

\thanks{This work was supported in part by the Ministry of Science and Technology of Taiwan under the grants 109-2811-M-032-508 and 108-2115-M-032-006-MY3, and by the CNRS-NCTS joint International Research Network ReaDiNet.}



\begin{abstract}
{We are concerned with the persistence of both predator and prey in a diffusive predator-prey system with a climate change effect, which is modeled by a spatial-temporal heterogeneity depending on a moving variable. Moreover, we consider both the cases of nonlocal and local dispersal. In both these situations, we first prove the existence of forced waves, which are positive stationary solutions in the moving frames of the climate change, of either front or pulse type. Then we address the persistence or extinction of the prey and the predator separately in various moving frames, and achieve a complete picture in the local diffusion case. We show that the survival of the species depends crucially on how the climate change speed compares with the minimal speed of some pulse type forced waves.}
\end{abstract}

\maketitle

\medskip

\section{Introduction}
\setcounter{equation}{0}

{Reaction-diffusion models and their nonlocal counterparts are commonly used in population dynamics to describe the behavior of the densities of concerned species,
including their survival or extinction, and the spatial spread of their habitat. More recently a lot of attention has been devoted to the effect of environmental heterogeneity which is ubiquitous in biological applications.
Here, specifically we would like to gain a better understanding of the consequences of a shifting heterogeneity in the context of prey-predator systems.
This is motivated by the modelling of global warming and its potentially dramatic consequences on ecological species, whose survival hinges on their ability to adapt and migrate according to these environmental changes.}

For the case of a single species, the following reaction-diffusion model has been proposed and attracted a lot of attention in the mathematical community:
\beaa
u_t(x,t)=du_{xx}(x,t)+u(x,t)f(x-st,u(x,t)),\; x\in\bR,\, t>0,
\eeaa
in which the function $f$ models the {climate change, which depends on a moving variable} with a positive speed $s$.
Here the species move with the standard random diffusion and the positive constant $d$ stands for the diffusion coefficient.
A typical example of $f$ is given by
\beaa
f(x-st,u(x,t))=\alpha(x-st)-u(x,t)
\eeaa
{for some function $\alpha$ which typically changes sign. In the set $\{z\mid\alpha(z)>0\}$, the linear growth rate is positive and one may refer to this set as the favorable region; similarly, the set~$\{ z \mid \alpha (z) <0\}$ is unfavorable to the species.}

{The goal is usually to derive criteria for the large-time survival or extinction of the species, and in the former case to understand its propagation. In this prospect, one of the main attentions is paid to the existence of a traveling wave solution with speed $s$, {\it the forced wave}, which is a positive stationary solution in the moving frame of the climate change. For this, we refer the reader to \cite{bd09} for the case of a bounded favorable region, where a dichotomy with respect to the climate change speed was established, as well as to~\cite{flw16,hz2017} for an unbounded favorable region, and~\cite{bf18} for some general KPP type nonlinearity. Forced waves and large-time behaviour of solutions of the Cauchy problem has also been studied in higher dimension~\cite{br08,br09}, including in a monostable (not necessarily KPP) case~\cite{bg19} where the dichotomy between extinction and survival also hinges on the size of the initial data. Moreover, propagation of the species in
 an unbounded favorable zone has been addressed in~\cite{lbsf2014}, and in time-periodic shifting habitat in~\cite{fpz21,v15}.

As far as systems involving several species are concerned, the existence and asymptotics of forced wave solutions is studied for a cooperative model in~\cite{ywl19}. For a Lotka-Volterra type competition model, the existence of forced waves was shown by Dong et al.~\cite{dll21}, and the persistence and extinction of species were established in~\cite{ywz19, zwy17}. Forced waves and gap formation in a competition model was also studied when the species' favorable habitats shift with opposite directions \cite{bdd14}.
Lastly, we also refer to some literature considering forced waves in domains with a free boundary \cite{dwz18,hhd20,hhsd20,ld17}.}

From the modelling point of view, it is sometimes relevant to replace the standard diffusion by a nonlocal dispersal which accounts for the long-range motion of some individuals. Concerning the study of the effects of climate change in this context, we refer the reader to some literature on the scalar equation which adresses both the persistence and extinction of the solution, as well as the existence and stability of the forced wave when the favorable zone is bounded~\cite{coville2020,LSZ}. A key point is the spectral analysis of an integro-differential operator; see~\cite{CovilleHamel} for related results. For the case when the favorable zone is unbounded, we refer to~\cite{lwz18} for the persistence of the species and the existence of the forced wave, and to~\cite{wz19} for the uniqueness and stability of the forced wave. We also point to~\cite{zz20} for propagation in a time-periodic shifting habitat. For a 2-species competition system, we refer to \cite{wwz19} for spatial-temporal dynamics and \cite{ww21} for the gap formation. In particular, it was shown in~\cite{lwz18} that there is a critical value for the climate change speed, below which the species manages to persist and above which it goes to extinction.

Yet much less is done for predator-prey systems, where new difficulties typically arise from a lack of a comparison principle. Therefore, in this paper, we consider the following diffusive predator-prey model with nonlocal dispersal
\begin{equation}\label{pp}
\begin{cases}
u_t(x,t)=d_{1}\bNu(x,t)+r_{1}u(x,t)[{\alpha(x-st)}-u(x,t)-av(x,t)], \; x\in\bR,\, t>0,\\
v_t(x,t)=d_{2}\bNv(x,t)+r_{2}v(x,t)[-1+bu(x,t)-v(x,t)], \; x\in\bR,\, t>0,%
\end{cases}
\end{equation}
where the unknown functions $u,v$ {respectively} stand for the population densities of prey and predator species at position $x$ and time $t$.
Parameters $d_1,d_2,r_1,r_2,a,b$ are positive and represent the diffusion coefficients, intrinsic growth rates, predation rate and conversion rate, respectively.
{As in the scalar equation case, the given positive constant $s$ denotes the climate change speed.}

Moreover, $\bNu(x,t)$ and $\bNv(x,t)$ formulate the spatial nonlocal dispersal of individuals and are given by
\beaa
\bNu(x,t) &:=& (J_1\ast u)(x,t)-u(x,t)=\int_{\mathbb{R}}J_{1}(x-y)u(y,t)dy-u(x,t), \\
\bNv(x,t) &:=& (J_2\ast v)(x,t)-v(x,t)=\int_{\mathbb{R}}J_{2}(x-y)v(y,t)dy-v(x,t),
\eeaa
in which $J_{i}$, $i=1,2$, are probability kernel functions satisfying the following conditions:
\begin{enumerate}
\item[(J1)] $J_i$ is nonnegative, continuous and compactly supported in $\bR$;

\item[(J2)] $\int_{\mathbb{R}}J_i(y)dy=1$ and $J_i(y)=J_i(-y)$ for all $y\in \mathbb{R}$.
\end{enumerate}

{Alternatively we will also consider the more classical case of a standard diffusion, that is
\begin{equation}\label{spp}
\begin{cases}
u_t(x,t)=d_{1} u_{xx}(x,t)+r_{1}u(x,t)[{\alpha(x-st)}-u(x,t)-av(x,t)], \; x\in\bR,\, t>0,\\
v_t(x,t)=d_{2} v_{xx}(x,t)+r_{2}v(x,t)[-1+bu(x,t)-v(x,t)], \; x\in\bR,\, t>0,%
\end{cases}
\end{equation}
where again the parameters $d_i$, $r_i$, $a$, $b$ are positive constants. Furthermore, in both local and nonlocal frameworks, the function $\alpha (\cdot)$ models the climate change which depends on a shifting variable,
and throughout the paper we assume that it satisfies the following properties:
\begin{enumerate}
\item[($\alpha$1)] $\alpha$ is continuous and nondecreasing in $\bR$;
\item[($\alpha$2)] $-\infty<\alpha(-\infty)<0<\alpha(\infty)<\infty$; furthermore, we choose $\alpha (\infty) =1$ without loss of generality (up to a rescaling).
\end{enumerate}
This means that the environment is favorable to the prey ahead of the climate change, then gradually deteriorates until it becomes hostile to the species.}\medskip

\section{Main results}\label{sec:results}

We are mainly concerned with the {question of persistence of both species, predator and prey, depending on the value of climate change speed~$s>0$.}

{On top of the previous assumptions, we also impose that
$$b >1,$$
which means that the predator population increases when the prey is at the maximal capacity. One may indeed check by comparison principles that the $v(\cdot,t) \to 0$ as $t \to +\infty$, uniformly in space, when $b < 1$.
In such a case \eqref{pp} formally reduces to a scalar equation which has already been studied in~\cite{lwz18}; see also \cite{lbsf2014} for the local case.}
Note that the intrinsic growth rate of the predator is assumed to be negative. Therefore, the predator cannot survive without the feeding prey resource.


\subsection{Forced waves}

Here we focus on the problem with nonlocal dispersal, though the system with local diffusion can be handled similarly. In the homogeneous case, solutions typically converge to a spatially constant stable steady state in the large time. However, such spatially constant steady states (aside from the trivial steady state 0)
no longer exist in \eqref{pp} due to the spatial heterogeneity in the term $\alpha$.
Moreover, even if the prey survives, we cannot expect it to persist in the part of the environment where its intrinsic growth rate, which is precisely the function $r_1\alpha$, is negative.

This leads us to introduce the notion of a \textit{forced wave}, which is a positive solution of \eqref{pp} and is stationary in the moving frame with the speed $s$ of the climate,
i.e. it is of the form $(u,v)(x,t)=(\hat\phi,\hat\psi)(\xi)$, $\xi:=x-st$. The functions $\{\hat\phi,\hat\psi\}$ (the wave profiles) then satisfy
\be\label{TWS0}
\begin{cases}
-s\hat\phi'(\xi)=d_1\mathcal{N}_1[\hat\phi](\xi)+r_1\hat\phi(\xi)[\alpha(\xi)-\hat\phi(\xi)-a\hat\psi(\xi)],\;\xi\in\bR,\\
-s\hat\psi'(\xi)=d_2\mathcal{N}_2[\hat\psi](\xi)+r_2\hat\psi(\xi)[-1+b\hat\phi(\xi)-\hat\psi(\xi)],\;\xi\in\bR,
\end{cases}
\ee
where
\beaa
\mathcal{N}_i[g](\xi):=\int_{\bR}J_i(\xi-y)g(y)dy-g(\xi),\; i=1,2.
\eeaa

We are interested in the following two different types of forced wave solutions. The first one is the front type, namely, a solution $(\hat\phi,\hat\psi)$ of \eqref{TWS0} such that
\be\label{bc0}
{(\hat\phi,\hat\psi)(-\infty)=(0,0),\;
(\hat\phi,\hat\psi)(\infty)=(u_*,v_*):=\left(\frac{1+a}{1+ab},\frac{b-1}{1+ab}\right),}
\ee
where $(u_*,v_*)$ is the unique constant co-existence state of \eqref{pp} with $\alpha\equiv\alpha(\infty)=1$.
This is in some sense corresponds to the best outcome scenario where both species persist ahead of the climate change, which typically arise when the initial conditions do not decay at infinity.

Another forced wave is the mixed front-pulse type, i.e., a solution $(\hat\phi,\hat\psi)$ of \eqref{TWS0} such that
\be\label{bc1}
(\hat\phi,\hat\psi)(-\infty)=(0,0),\; (\hat\phi,\hat\psi)(\infty)=(1,0),
\ee
where $(1,0)$ is the predator-free state with prey in its maximal capacity $1$. In this second type, the predator component of the forced wave is a pulse, and it corresponds to a non-trivial threshold between persistence and extinction of this species. In particular, we will see below that the mixed front-pulse forced waves only exist above some critical speed $s_*$ (see~\eqref{maxspeed_predator} below). This will also turn out to be a crucial parameter for the large-time persistence of the predator in the Cauchy problem; we refer to Subsections~\ref{subsec:spread_nonlocal} and~\ref{subsec:spread_local} below for more details.

We point out that there also exists a third type, which we may refer to as the pulse type forced wave, and which is a positive solution $(\hat \phi , 0)$ of \eqref{TWS0} such that
\begin{equation*}
\hat\phi (-\infty)= \hat \phi (\infty) = 0.
\end{equation*}
As far as the pulse type forced wave is concerned, the system {\eqref{pp} is reduced} to a scalar equation which has been studied for instance in~\cite{coville2020,lwz18}.
It is expected that a pulse type forced wave exists if and only if
$$s \geq s^{*} :=\inf_{\lambda>0}\frac{d_1(\int_\bR{J_1(y)e^{\lambda y}dy}-1)+r_1}{\lambda}, $$
where $s^*$ will also arise in our large-time persistence and spreading results for the prey in Subsections~\ref{subsec:spread_nonlocal} and~\ref{subsec:spread_local}.
{Since this result can be obtained by the same method as for the mixed front-pulse type waves, we omit its detailed proof in this paper.}


For the existence of front type forced waves for system \eqref{pp}, {our main result reads as follows.}
\begin{theorem}\label{th:forced}
Suppose that
\begin{equation}\label{ab2}
b>1, \quad ab<1.
\end{equation}
Then there exists a positive solution $(\hat\phi,\hat\psi)$ of \eqref{TWS0} and \eqref{bc0}.
\end{theorem}

For the mixed front-pulse type forced waves, since $b >1$, the quantity
\begin{equation}\label{maxspeed_predator}
s_{*} :=\inf_{\lambda>0}\frac{d_2 (\int_\bR{J_2(y)e^{\lambda y}dy}-1)+r_2 (b-1)}{\lambda}
\end{equation}
is well-defined. In fact, it is the spreading speed of the predator population when the density of preys is fixed to its maximal capacity~$1$; cf. \cite{DGLP19}.
Then we have:
\begin{theorem}\label{th:forced2}
Assume that $b>1$. Suppose that, in addition to {\rm ($\alpha$1)} and {\rm ($\alpha$2)}, $\alpha$ satisfies
\begin{enumerate}
{\item[{\rm ($\alpha$3)}]  $\alpha(\infty)-\alpha(z)\le Ce^{-\rho z}$ for all large $z$ for some positive constants $C$ and $\rho$.}
\end{enumerate}
Then there exists a positive solution $(\hat\phi,\hat\psi)$ of \eqref{TWS0} and \eqref{bc1} if and only if $s\ge s_*$.
\end{theorem}

For the local diffusion case,  the wave profiles $ \{\hat\phi,\hat\psi\}$ satisfy
\be\label{TWS1}
\begin{cases}
-s\hat\phi'(\xi)=d_1\hat\phi''(\xi)+r_1\hat\phi(\xi)[\alpha(\xi)-\hat\phi(\xi)-a\hat\psi(\xi)],\;\xi\in\bR,\\
-s\hat\psi'(\xi)=d_2\hat\psi''(\xi)+r_2\hat\psi(\xi)[-1+b\hat\phi(\xi)-\hat\psi(\xi)],\;\xi\in\bR.
\end{cases}
\ee
Then the same results as Theorems~\ref{th:forced} and \ref{th:forced2} hold for system \eqref{TWS1}, with $s_*:=2\sqrt{d_2r_2(b-1)}$ instead of~\eqref{maxspeed_predator}.


\subsection{Spreading behaviours: nonlocal case}\label{subsec:spread_nonlocal}

We now turn to the question of the large time behaviour of solutions of the Cauchy problems \eqref{pp} and \eqref{spp},
and more precisely to the question of the persistence or extinction of either species.
To study the spatial-temporal dynamics of both species, we consider the solution $(u,v)$ of \eqref{pp} supplemented with the initial condition
\be\label{pp-ic}
u(x,0)=u_0(x),\; v(x,0)=v_0(x),\;x\in\bR,
\ee
where $u_0$ and $v_0$ are nonnegative bounded continuous functions.

When the environment is homogeneous and there is no predator, that is say $\alpha \equiv 1$ and $v \equiv 0$ in the $u$-equation of \eqref{pp},
then the spreading speed of the population prey is given by the quantity
\begin{equation}\label{maxspeed_prey}
s^{*} :=\inf_{\lambda>0}\frac{d_1(\int_\bR{J_1(y)e^{\lambda y}dy}-1)+r_1}{\lambda}.
\end{equation}
More precisely, when the initial condition is nontrivial and compactly supported, the solution of
$$u_t(x,t)=d_{1}\bNu(x,t)+r_{1}u(x,t)[1-u(x,t)], \; x\in\bR,\, t>0,$$
converges to $1$ locally uniformly in any moving frame with speed less than $s^*$, and to $0$ in moving frames with speed larger than $s^*$, and furthermore the value $s^*$ is also the minimal speed for traveling wave solutions (cf. \cite{jz09,w82}). It is then natural to expect here that whether the prey manages to keep pace with the climate change will be determined by whether the climate change speed $s$ is faster or slower than this $s^*$.

{Similarly, we expect that the persistence of the predator will be (in part) determined by the comparison of the values~$s$ and~$s_*$. We recall that $s_*$ was defined from \eqref{maxspeed_predator} by
$$
s_{*} :=\inf_{\lambda>0}\frac{d_2 (\int_\bR{J_2(y)e^{\lambda y}dy}-1)+r_2 (b-1)}{\lambda},
$$
which is the spreading speed of the predator when the density of prey is fixed to its maximal capacity~1, and which has been established in Theorem~\ref{th:forced2} to also be the minimal speed for mixed front-pulse forced waves.}

Consistently with the above discussion, our results show that in order to survive both species must keep pace with the speed of climate change.
Moreover, both predator and prey must move their habitats {accordingly to the shift of the ``favorable'' environment, which is ahead of the moving frame with speed $s$.}

More precisely, {we first consider the case when the prey is faster than the predator,
in the sense that its maximal speed $s^*$ (i.e. in the ``favorable'' environment and without prey) is larger than the maximal speed $s_*$ of the predator (i.e. when the prey density is at saturation).}

\begin{theorem}\label{spread_fastprey}
Assume that $b>1$ and $s^*  > s_*$. Let $(u,v)$ be the solution of {\eqref{pp} and \eqref{pp-ic}}, where $0 \leq u_0 \leq 1$, $0 \leq v_0 \leq b-1$ are both nontrivial, continuous and compactly supported.
\begin{itemize}
\item If $s > s^*$, then
 $$\lim_{ t \to +\infty} \sup_{x \in \R} [u(x,t) + v(x,t)] = 0.$$
\item If $s \in (s_* ,s^*)$, then for any $\varepsilon >0$,
\beaa
&&\lim_{t \to +\infty} \sup_{x \in \R} v (x,t) = 0,\\
&&\lim_{t \to +\infty} \{  \sup_{ x \leq (s-\varepsilon) t} u(x,t) + \sup_{x \geq (s^*+ \varepsilon) t} u (x,t) \} =0,\\
&&\lim_{t \to +\infty}  \sup_{ (s + \varepsilon) t \leq x \leq (s^* -\varepsilon) t} | u(x,t) -1|  =0.
\eeaa
\item If $s < s_*$, then for any $\varepsilon >0$,
\beaa
&&\lim_{t \to +\infty} \{ \sup_{x \leq ( s - \varepsilon) t } v (x,t) + \sup_{x \geq (s_* + \varepsilon) t} v(x,t) \} = 0,\\
&&\lim_{t \to +\infty} \{ \sup_{x \leq ( s - \varepsilon) t } u (x,t) + \sup_{x \geq (s^* + \varepsilon) t} u (x,t) \} = 0,\\
&&\lim_{t \to +\infty}  \sup_{ (s_* + \varepsilon) t \leq x \leq (s^* -\varepsilon) t} | u(x,t) -1|  =0.
\eeaa
\end{itemize}
\end{theorem}


{Next, we consider the situation when $s^* \leq s_*$, i.e., the predator is faster than the prey. Here the situation is slightly different because the prey may never outrace the predator.
Conversely, the predator cannot go beyond the habitat of the prey, and in particular we expect that $s^*$ is the critical climate change speed for both species.}

\begin{theorem}\label{spread_fastpred}
Assume that $b>1$ and $s^*  \leq s_*$. Let $(u,v)$ be the solution of \eqref{pp} and \eqref{pp-ic}, where $0 \leq u_0 \leq 1$, $0 \leq v_0 \leq b-1$ are both nontrivial, continuous and compactly supported.
\begin{itemize}
\item If $s > s^*$, then
$$\lim_{ t \to +\infty} \sup_{x \in \R} [u(x,t) + v(x,t)] = 0.$$
\item If $s < s^*$, then for any $\varepsilon >0$,
\beaa
&&\lim_{t \to +\infty} \{ \sup_{x \leq ( s - \varepsilon) t } u (x,t) + \sup_{x \geq (s^* + \varepsilon) t} u (x,t) \} = 0,\\
&&\lim_{t \to +\infty} \{ \sup_{x \leq ( s - \varepsilon) t }v (x,t) + \sup_{x \geq (s^* + \varepsilon) t} v (x,t) \} = 0,
\eeaa
\end{itemize}
\end{theorem}

Theorems~\ref{spread_fastprey} and~\ref{spread_fastpred} leave open what happens between the moving frames with speeds~$s$ and $\min \{ s_*, s^* \}$. We expect that both species should always persist there, and we will prove some partial results in that direction in Theorems~\ref{upos} and~\ref{vpos} below. A more complete picture will be provided in the case of local diffusion.

For the next results, we need to define the following two quantities
\begin{align*}
s^{**}&:=\inf_{\lambda>0}\frac{d_1(\int_\bR{J_1(y)e^{\lambda y}dy}-1)+r_1[1-a(b-1)]}{\lambda}, \\
s_{**}&:=\inf_{\lambda>0}\frac{d_2(\int_\bR{J_2(y)e^{\lambda y}dy}-1)+r_2(b-1)(1-ab)}{\lambda}.
\end{align*}
{The former can be understood as the speed of the prey in the favorable environment when there is maximal amount of predator, and the latter as the speed of the predator when there is a minimal amount of prey. Notice that these} are well-defined when $b>1$ and $ab<1$.

\begin{theorem}\label{upos}
Assume that $b>1$ and $ab<1$. 
If $s<s^{**}$, then for any $\ep\in(0,(s^{**}-s)/2)$, there is a positive constant $\kappa$ such that
\beaa
\liminf_{t\to\infty} \, \{\inf_{(s+\ep)t\le x\le (s^{**}-\ep)t}u(x,t)\}\ge\kappa
\eeaa
for any solution of $(u,v)$ of \eqref{pp} and~\eqref{pp-ic} where $0 \leq \not \equiv u_0 \leq 1$, $0 \leq v_0 \leq b-1$ are continuous.
\end{theorem}

\begin{theorem}\label{vpos}
Assume that $b>1$ and $ab<1$. Set $\underline{s}^*:=\min\{s^{**},s_{**}\}$ and suppose that $s<\underline{s}^*$.
 Then for any solution $(u,v)$ of \eqref{pp} and~\eqref{pp-ic} where $0 \leq  u_0 \leq 1$, $0 \leq v_0 \leq b-1$ are both nontrivial and continuous, we have
\beaa
\liminf_{t\to\infty} \, \{\inf_{(s+ \ep)t\le x\le (\underline{s}^*-\ep)t} v(x,t)\}>0,
\eeaa
for any $\ep\in(0,(\underline{s}^*-s)/2)$.
\end{theorem}


\subsection{Spreading behaviours: local case}\label{subsec:spread_local}
As in the nonlocal framework, we expect that for the standard diffusion case~\eqref{spp}, the persistence of the prey and predator shall be determined by comparing the climate shifting speed $s$ with
the maximal speed $s^*$ of prey and the maximal speed $s_*$ of predator, respectively.
Here the maximal speeds are defined by
\be\label{sspeed}
s^*:=2\sqrt{d_1r_1},\;
s_{*}:=2\sqrt{d_2r_2(b-1)},
\ee
and similarly as before these are the spreading speeds of the solutions of the $u$-equation when $\alpha \equiv 1$ and $v \equiv 0$, and of the $v$-equation when $u \equiv 1$, respectively.

Then the same results as Theorems~\ref{spread_fastprey} and~\ref{spread_fastpred} hold for the solution $(u,v)$ of \eqref{spp} and~\eqref{pp-ic}. For the sake of conciseness, we do not re-write these results here, but refer to \S5 for some details. Moreover, the same result as Theorem~\ref{upos} holds for any solution $(u,v)$ of \eqref{spp}, by re-defining $s^{**}:=2\sqrt{d_1r_1[1-a(b-1)]}$, when $b>1$ and $ab<1$.

{However, in the local diffusion case we are actually able to deal with the persistence in all the remaining moving frames with speeds {between $s$ and
$$\uns := \min \{ s^*, s_*\}>s.$$}
Indeed, our last main result is the following persistence theorem:
\begin{theorem}\label{svpos}
Assume that $b>1$. Let $(u,v)$ be the solution of \eqref{spp} and \eqref{pp-ic}, where $0 \leq u_0 \leq 1$, $0 \leq v_0 \leq b-1$ are both nontrivial, continuous and compactly supported. If $s < \uns$, then for any $\eta\in(0,(\uns-s)/2)$,
$$\lim_{t \to +\infty} \{\sup_{(s+\eta)t \leq x\leq ( \uns - \eta) t} (|  u(x,t)- u_* | + |v(x,t) - v_*|) \}=  0, $$
where {$(u_*,v_*)$ is the positive co-existence steady state defined in \eqref{bc0}.}
\end{theorem}
The main idea of the proof of Theorem~\ref{svpos} is motivated by a method used in \cite{dggs2021,dgm} which is different from that for the proof of Theorems~\ref{upos} and~\ref{vpos}. It strongly relies on parabolic estimates and the compactness of the set of solutions with bounded initial data, which is a much more difficult issue in the nonlocal diffusion framework. This result offers a clearer picture for the large-time dynamics of the solution in the local diffusion case, as it shows that both species persist in the intermediate moving ranges and also converge to the co-existence steady state.}

\bigskip

\paragraph{\textbf{Plan of the paper.}}
The rest of this paper is organized as follows. In \S3, we provide a proof of Theorems~\ref{th:forced} and \ref{th:forced2}, that is the existence and non-existence of forced waves for \eqref{pp}. Then we study the spatial-temporal dynamics of \eqref{pp} in \S4. As we will see, our method for the nonlocal dispersal case can be easily applied to the standard diffusion system. Finally, in \S5, we provide some detailed proofs for the refined results in the standard diffusion case.

\section{Forced waves: the nonlocal case}
\setcounter{equation}{0}

In this section, we shall derive the existence of forced waves for \eqref{pp}.
Following \cite{lwz18}, we set $(\phi,\psi)(z):=(\hat\phi,\hat\psi)(-z)$ for a solution $(\hat\phi,\hat\psi)$ of \eqref{TWS0}.
Then $(\phi,\psi)$ satisfies
\be\label{TWS}
\begin{cases}
s\phi'(z)=d_1\mathcal{N}_1[\phi](z)+r_1\phi(z)[\alpha(-z)-\phi(z)-a\psi(z)],\;z\in\bR,\\
s\psi'(z)=d_2\mathcal{N}_2[\psi](z)+r_2\psi(z)[-1+b\phi(z)-\psi(z)],\;z\in\bR.
\end{cases}
\ee
Here we have used
\beaa
{\mathcal{N}_1[\hat\phi](-z)=\mathcal{N}_1[\phi](z),\quad\mathcal{N}_2[\hat\psi](-z)=\mathcal{N}_2[\psi](z),}
\eeaa
by using the symmetry of $J_i$, $i=1,2$.
Note that condition \eqref{bc0} becomes
\be\label{bc}
(\phi,\psi)(-\infty)=(u_*,v_*),\quad(\phi,\psi)(\infty)=(0,0).
\ee

Now we introduce the following {notion of} generalized upper-lower solutions of \eqref{TWS}.

\begin{definition}\label{lus}
Continuous functions $(\overline{\phi },\overline{\psi })$ and $(\underline{\phi},\underline{\psi})$
are called a pair of upper and lower solutions of \eqref{TWS} if
$\underline{\phi}(z)\le \overline{\phi}(z)$, $\underline{\psi}(z))\le\overline{\psi}(z))$ for all $z \in \mathbb{R}$
and the following inequalities
\begin{eqnarray}
&&s\overline{\phi}^{\prime}(z)\ge d_{1}\mathcal{N}_{1}[\overline{\phi}](z)+r_{1}\overline{\phi}(z)[\alpha(-z)-\overline{\phi}(z)-a\underline{\psi}(z)],  \label{u1} \\
&&s\overline{\psi}^{\prime}(z)\ge d_{2}\mathcal{N}_{2}[\overline{\psi}](z)+r_{2}\overline{\psi}(z)[-1+b\overline{\phi }(z)-\overline{\psi}(z)],  \label{u2} \\
&&s\underline{\phi}^{\prime}(z)\le d_{1}\mathcal{N}_{1}[\underline{\phi}](z)+r_{1}\underline{\phi}(z)[\alpha(-z)-\underline{\phi}(z)- a\overline{\psi}(z)],  \label{l1} \\
&&s\underline{\psi}^{\prime}(z)\le d_{2}\mathcal{N}_{2}[\underline{\psi}](z)+r_{2}\underline{\psi}(z)[-1+ b\underline{\phi}(z)-\underline{\psi}(z)]  \label{l2}
\end{eqnarray}%
hold for all $z \in \mathbb{R}\setminus E$ for some finite subset $E$ of $\mathbb{R}$.
\end{definition}


Then, from Schauder's fixed-point theorem, we have the following lemma.

\begin{lemma}\label{luslem}
Let $s>0$ be given. Let $(\ophi,\opsi)$ and $(\uphi,\upsi)$ be a pair of upper and lower solutions of \eqref{TWS}, {and further assume that
$$0 \leq \uphi \leq \ophi \leq 1, \qquad 0 \leq \upsi \leq \opsi \leq b-1.$$
Then} \eqref{TWS} admits a solution $(\phi,\psi)$ such that
$\uphi(\xi)\le \phi(\xi)\le \ophi(\xi)$ and $\upsi(\xi)\le\psi(\xi)\le\opsi(\xi)$ for all $\xi\in\mathbb{R}$.
\end{lemma}

\begin{proof}
First, let $X$ be the space of all uniformly continuous and bounded functions defined in $\mathbb{R}$. Then $X$ is a Banach space equipped with the sup-norm.
Furthermore, we let
$$\tilde{X}:=\{(w_1,w_2)\in X^2:0\le w_1(x)\le 1,0\le w_2(x)\le b-1,\,\forall\, x\in\mathbb{R} \}.$$

Next, we consider the nonlinear operators $F_1$ and $F_2$ defined on $\tilde{X}$ by
\begin{align*}
F_1(\phi,\psi)(z)&:=\beta\phi(z) + d_1\mathcal{N}_1[\phi](z)+r_1\phi(z)[\alpha(-z)-\phi(z)-a\psi(z)],\;z\in\bR,\\
F_2(\phi,\psi)(z)&:=\beta\psi(z) + d_2\mathcal{N}_2[\psi](z)+r_2\psi(z)[-1+b\phi(z)-\psi(z)],\;z\in\bR,
\end{align*}
for some positive constant $\beta$ satisfying
\be\label{beta-con}
{\beta>\max\{d_1 + r_1[-\alpha(-\infty)+2+a(b-1)], d_2+r_2(2b-1)\}.}
\ee
We also define the following operators:
\begin{align*}
P_1(\phi,\psi)(z)&:=\frac{1}{s}\int_{-\infty}^z{\exp\left(-\frac{\beta(z-y)}{s}\right)F_1(\phi,\psi)(y)dy},\;z\in\bR,\\
P_2(\phi,\psi)(z)&:=\frac{1}{s}\int_{-\infty}^z{\exp\left(-\frac{\beta(z-y)}{s}\right)F_2(\phi,\psi)(y)dy},\;z\in\bR.
\end{align*}
Set $P=(P_1,P_2)$. Then $P:\tilde{X} \to X^2$ and it is easy to check that a fixed point of $P$ is a solution of \eqref{TWS}.
Therefore, it remains to show that $P$ has a fixed point.


Let $\mu>0$ be a constant {such that $\mu<{\beta}/{s}$} and let
$$|(\phi,\psi)|_\mu=\sup_{z\in\bR}\{\max(|\phi(z)|,|\psi(z)|)e^{-\mu|z|}\},\;(\phi,\psi)\in \tilde{X}.$$
Then it is easy to check that $(\tilde{X},|\cdot|_\mu)$ is a Banach space. Moreover, the set
$$\Gamma:=\{(\phi,\psi)\in \tilde{X}:\uphi\le\phi\le\ophi,\upsi\le\psi\le\opsi\}$$
is a non-empty convex, closed and bounded set in $(\tilde{X},|\cdot|_\mu)$.

Now, we show that $P$ maps $\Gamma$ into $\Gamma$. Let $(\phi,\psi)\in\Gamma$. Then, using \eqref{beta-con}, we can check that
\beaa
\mbox{$F_1(\phi,\psi)(z)\ge F_1(\uphi,\opsi)(z)$ for all $z\in\bR$.}
\eeaa
Thus we obtain $P_1(\uphi,\opsi)\le P_1(\phi,\psi)$.
On the other hand, by the definition of upper-lower solutions, we have
\begin{align*}
P_1(\uphi,\opsi)(z)&=\frac{1}{s}\int_{-\infty}^z{\exp\left(-\frac{\beta(z-y)}{s}\right)F_1(\uphi,\opsi)(y)dy}\\
&\ge\int_{-\infty}^z{\exp\left(-\frac{\beta(z-y)}{s}\right)\left(\uphi'(y)+\frac{\beta}{s}\uphi(y)\right)dy}=\uphi(z)
\end{align*}
for all $z\in\bR$. Hence $P_1(\phi,\psi)\ge\uphi$.
Similarly, we have
\begin{align*}
P_1(\phi,\psi)\le P_1(\ophi,\upsi),\quad P_2(\uphi,\upsi)\le P_2(\phi,\psi)\le P_2(\ophi,\opsi),
\end{align*}
by the choice of $\beta$ in \eqref{beta-con}.
Hence we obtain that $P(\Gamma)\subset \Gamma$.
%

Finally, by the choice of $\mu$, we can show that the mapping $P:\Gamma\to\Gamma$ is completely continuous with respect to the norm $|\cdot|_\mu$.
Since a proof of the complete continuity of $P$ can be found for instance in~\cite{hlr05,m01}, we omit it here. Hence it follows from Schauder's fixed-point theorem that $P$ has a fixed point in $\Gamma$.
This completes the proof of the lemma.
\end{proof}

\subsection{Existence of wave profiles}

This section is devoted to the existence of a solution to~\eqref{TWS}. {To construct the front type and mixed front-pulse type, we need to introduce different pairs of upper and lower solutions of~\eqref{TWS}, which will in turn ensure, in the next subsection, the correct asymptotics at infinity.}

First, for the front type waves, we assume here, in addition to $b>1$, the condition $ab<1$.
{It follows that $a (b-1) <1$, which ensures that, even when there are many predators ($v \equiv b-1$), the prey may still survive at least in an environment without climate change.}

In particular, it follows from \cite[Theorem 4.5]{lwz18} that there is a {{non-increasing} positive} function $\uphi$ such that
\be\label{uphi}
s\uphi'(z)=d_1\mathcal{N}_1[\uphi](z) + r_1\uphi(z)[\alpha(-z)-a(b-1)-\uphi(z)],\;\;z\in\bR,
\ee
 and
\be\label{uphi_lim}
\lim_{z\to -\infty}\uphi(z)=1-a(b-1)>0, \;\;\lim_{z\to\infty}\uphi(z)=0.
\ee
Furthermore, since also $b[1-a(b-1)]>1$, it follows from \cite[Theorem 4.5]{lwz18} again that there exists a {{non-increasing} positive} function $\upsi$ such that
\be\label{upsi}
s\upsi'(z)=d_2\mathcal{N}_2 [\upsi ](z) + r_2\upsi(z)[-1+b\uphi(z)-\upsi(z)],\;\;z\in\bR,
\ee
 and
\be\label{upsi_lim}
\lim_{z\to -\infty}\upsi(z)=-1+b[1-a(b-1)]>0, \;\;\lim_{z\to\infty}\upsi(z)=0.
\ee

\begin{lemma}\label{existence}
Suppose that \eqref{ab2} holds. Then there exists a solution $(\phi,\psi)$ of \eqref{TWS} such that {$0< \uphi\le\phi\le 1$ and $0<\upsi\leq \psi\leq b-1$} in $\bR$, where $\uphi$ and $\upsi$ are solutions of \eqref{uphi}-\eqref{uphi_lim} and \eqref{upsi}-\eqref{upsi_lim}, respectively.
\end{lemma}
\begin{proof}
Let $(\ophi,\opsi)=(1,b-1)$.
It is clear that $\uphi\le 1$ and $\upsi\le b-1$.
We only need to check that $(\uphi,\upsi)$ and $(\ophi,\opsi)$ satisfies (\ref{u1})-(\ref{l2}).

Since {$\mathcal{N}_1 [\ophi]=\int_\bR{J_1(y)dy}-1=0$} and $\alpha(-z)\le1$ for $z\in\bR$, we have that
\begin{align*}
d_1\mathcal{N}_1 \left[ \ophi \right](z)+r_1\ophi(z)[\alpha(-z)-\ophi(z)-a\upsi(z)]&\le r_1[\alpha(-z)-1]\le 0 = s\ophi'(z),\;\forall\, z\in\bR,
\end{align*}
and so \eqref{u1} holds.
Similarly, (\ref{u2}) holds because
\begin{align*}
 d_2\mathcal{N}_2 \left[\opsi \right](z)+r_2\opsi(z)[-1+b\ophi(z)-\opsi(z)]=r_2(b-1)[-1+b -(b-1)]=0=s\opsi'(z).
\end{align*}
Also, it follows from (\ref{uphi}) and (\ref{upsi}) that
\begin{align*}
s\underline{\phi}^{\prime}(z)&= d_{1}\mathcal{N}_{1}[\underline{\phi}](z)+r_{1}\underline{\phi}(z)[\alpha(-z)-a(b-1)-\underline{\phi}(z)]\\
&=d_{1}\mathcal{N}_{1}[\underline{\phi}](z)+r_{1}\underline{\phi}(z)[\alpha(-z) -a\overline{\psi}(z)-\underline{\phi}(z)],\;\forall\,z\in\bR,
\end{align*}
and
\begin{align*}
s\underline{\psi}^{\prime}(z)= d_{2}\mathcal{N}_{2}[\underline{\psi}](z)+r_{2}\underline{\psi}(z)[-1+ b\underline{\phi}(z)-\underline{\psi}(z)],\;\forall\,z\in\bR,
\end{align*}
in other words (\ref{l1}) and (\ref{l2}) hold.
Therefore, $(\uphi,\upsi)$ and $(1,b-1)$ are a pair of upper and lower solutions. The lemma is proved by applying Lemma~\ref{luslem}.
\end{proof}

Secondly, we deal with the mixed front-pulse type waves. To find a suitable pair of upper and lower solutions, we define
$$\Delta(\lambda,s):=d_2\left[\int_\bR{J_2(y)e^{\lambda y}dy} -1\right] + r_2(b-1)-s\lambda.$$
Then we have the following properties for $\Delta(\lambda,s)$:

(i) for $s>s_*$, $\Delta(\lambda,s)=0$ has two distinct positive roots $\ld_1,\ld_2$ with $\lambda_1<\lambda_2$ such that $\Delta(\lambda,s)<0$ if and only if $\lambda_1<\lambda<\lambda_2$;

(ii) for $s=s_*$, $\Delta(\lambda,s)=0$ has a double root at $\ld=\ld_*$ such that $\Delta(\ld,s)>0$ for all $\ld\neq\ld_*$;

(iii) for $s < s^*$, $\Delta(\lambda,s)>0$ for all $\lambda\ge0$.

{\noindent These properties follow from the straightforward facts that $\lambda \mapsto \Delta (\lambda,s)$ is convex, $\Delta (\lambda,s)$ goes to $+\infty$ as $\lambda \to +\infty$, and $\Delta (0,s) = r_2 (b -1) >0$.}

{\bf Case 1: $s>s_*$.} We first define
\begin{equation}\label{supsub_s>s*}
\left\{
\begin{array}{ll}
\ophi(z)\equiv1,& \opsi(z)=\min\{b-1,e^{\lambda_1z}\},  \vspace{3pt}\\
 \uphi (z) = \max\{0,1-\eta e^{\mu z}\} ,& \upsi(z)=\max\{0,e^{\lambda_1z}-ke^{(\lambda_1+\mu)z}\},
\end{array}
\right.
\end{equation}
where $\mu\in(0,\ld_2-\ld_1)$ and $\eta, k>1$ are constants to be determined later.

We claim that $(\ophi,\opsi)$ and $(\uphi,\upsi)$ are a pair of upper and lower solutions of \eqref{TWS} for some constants $k>1$ and $\mu\in(0,\ld_2-\ld_1)$. Indeed, we can easily check that (\ref{u1}) holds. Next we turn to (\ref{u2}). We let $z_1$ {be} such that $e^{\lambda_1 z_1}=b-1$. Then
$$\opsi(z)=\begin{cases}
e^{\lambda_1 z},\;\;&z\le z_1,\\
b-1, &z\ge z_1,
\end{cases}$$
and it is clear that (\ref{u2}) holds for $z>z_1$.
For $z<z_1$, we compute
\beaa
&&d_2\mathcal{N}_2[\opsi](z)-s\opsi'(z)+r_2\opsi(z)[-1+b-\opsi(z)]\\
&\leq &d_2\left[\int_\bR{J_2(y)e^{\lambda_1(z-y)}dy}-\expl\right]-s\lambda_1\expl + r_2\expl(-1+b-\expl)\\
&\leq &\expl\left(d_2\Big[\int_\bR{J_2(y)e^{\lambda_1 y}dy}-1\Big]-s\lambda_1 + r_2(b-1)-r_2\expl\right)\\
&\leq &-r_2 e^{2\lambda_1 z}<0.
\eeaa
Hence (\ref{u2}) holds for all $z\neq z_1$.

{Next, we set
\beaa
\mu_0:=\min\{\ld_2-\ld_1,\ld_1,\rho\},
\eeaa
and let $\mu\in(0,\mu_0)$ be a constant such that
\be\label{mu}
A(\mu):=d_1\Big[\int_\bR{J_1(y)e^{\mu y}dy}-1\Big]-s\mu<0.
\ee
The existence of such $\mu$ follows from the facts $A(0)=0$ and $A'(0)=-s<0$.

Now note that $\uphi(z)=1-\eta e^{\mu z}$ for $z\le z_2=z_2(\eta)$, where $z_2$ is defined by $1=\eta e^{\mu z_2}$.
Since $z_2(\eta)\searrow -\infty$ as $\eta\to\infty$, we can choose $\eta$ large enough so that $z_2(\eta)\le z_1$ and
\beaa
\alpha(-z)\ge \alpha(\infty)-Ce^{\rho z}=1-Ce^{\rho z},\;\forall\, z\le z_2,
\eeaa
using the condition ($\alpha 3$). Then we check that \eqref{l1} holds for $z \neq z_2$. First, for $z<z_2<0$, we have
\beaa
&&d_1\mathcal{N}_1[\uphi](z)-s\uphi'(z) + r_1\uphi(z)[\alpha (-z) - a \opsi (z) -\uphi(z)]\\
&\ge&-\eta e^{\mu z}A(\mu) + r_1\uphi(z)[\eta e^{\mu z}-Ce^{\rho z}-ae^{\lambda_1 z}]\\
&\ge& e^{\mu z}\left\{-\eta A(\mu) + r_1\uphi(z)[\eta -Ce^{(\rho-\mu) z}-ae^{(\lambda_1-\mu) z}]\right\}>0,
\eeaa
using \eqref{mu}, $\mu<\mu_0$, and by choosing $\eta$ larger if necessary. Moreover it is straightforward that \eqref{l1} also holds for $z>z_2$.}

Lastly, with some suitable choice of $k$, we show that (\ref{l2}) holds for all $z\neq z_3$, where $z_3$ is defined by $e^{\mu z_3}=1/k$. In particular,
$$\upsi(z)=\begin{cases}
e^{\lambda_1 z}-k e^{(\ld_1+\mu) z},\;\;&z<z_3,\\
0, &z\ge z_3.
\end{cases}$$
We only need to deal with the case $z<z_3$, and choose $k\ge\eta$ such that
\be\label{klarge}
k\ge \frac{r_2(b\eta+1)}{-\Delta(\ld_1+\mu,s)},
\ee
which is well-defined thanks to $\Delta(\ld_1+\mu,s)<0$, by our choice of $\mu \in (0,\mu_0)$. Also, we have that $z_3 \le z_2$, hence we compute
\beaa
&&d_2\mathcal{N}_2[\upsi](z)-s\upsi'(z)+r_2 \upsi(z)[-1+b\uphi(z)-\upsi(z)]\\
&\ge&\expl\Big(d_2\Big[\int_\bR{J_2(y)e^{\lambda_1 y}dy}-1\Big]-s\lambda_1\Big)\\
&&\; -k e^{(\ld_1+\mu) z}\Big[d_2\Big[\int_\bR{J_2(y)e^{(\ld_1+\mu) y}dy}-1\Big]-s(\ld_1+\mu)\Big]\\
&&\quad +r_2 \upsi(z)[(b-1)-b\eta e^{\mu z}-\upsi(z)]\\
&=&-r_2(b-1)\expl-k e^{(\ld_1+\mu)z}[\Delta(\ld_1+\mu,s)-r_2(b-1)]\\
&&\quad +r_2(b-1) \upsi(z)-r_2 \upsi(z)[b\eta e^{\mu z}+\upsi(z)]\\
&\ge&e^{(\ld_1+\mu)z}\left\{-k\Delta(\ld_1+\mu,s)-r_2[b\eta \upsi(z) e^{-\ld_1 z}+\upsi^2(z) e^{-(\ld_1+\mu)z}]\right\}
\eeaa
for $z<z_3$.
Note that
\beaa
b\eta\upsi(z)e^{-\ld_1 z}+\upsi^2(z) e^{-(\ld_1+\mu)z}\le b\eta+e^{(\ld_1-\mu)z}\le b\eta+1\;\mbox{ for all $z<z_3$,}
\eeaa
using $\upsi(z)\le \expl$, $\mu<\ld_1$ and $z_3<0$.
Hence \eqref{l2} holds for $z<z_3$, due to \eqref{klarge}.
This proves that $(\ophi,\opsi)$ and $(\uphi,\upsi)$ are a pair of upper and lower solutions of \eqref{TWS}.

{\bf Case 2: $s=s_*$.} Motivated by \cite{DGLP19}, we consider
\beaa
\Psi(z):=-Lz e^{\lambda_*z},
\eeaa
where the constant $L\ge (b-1)\ld_* e$ is so that the maximum of the function $\Psi(-1/\ld_*)\ge b-1$. Hence there is $z_1\le -1/\ld_*$ such that $\Psi(z_1)=b-1$ and $0 < \Psi(z)<b-1$ for all $z<z_1$.
Then we define a non-decreasing function $\opsi$ by
\begin{equation}\label{supsub_s=s*0}
\opsi(z)=
\begin{cases}
\Psi(z),\; z<z_1,\\
b-1,\; z\ge z_1.
\end{cases}
\end{equation}
Also, we consider the functions
\begin{equation}\label{supsub_s=s*}
\left\{
\begin{array}{ll}
\ophi(z)\equiv1,& \vspace{3pt}\\
 \uphi (z) = \max\{0,1-\eta e^{\mu z}\} ,& \upsi(z) =\max\{0, [- Lz -q\sqrt{-z}]e^{\ld_* z}\}.
\end{array}
\right.
\end{equation}
We claim that $(\ophi,\opsi)$ and $(\uphi,\upsi)$ are a pair of upper and lower solutions of \eqref{TWS} with $s =s_*$ for some suitably chosen positive constants $\eta$ and $q$. For the reader's convenience, we provide below the detailed verifications.

As before, it is clear that \eqref{u1} holds for all $z\in\bR$. For \eqref{u2}, it suffices to consider the case when $z < z_1$. Recall that $\ld_*$ is a double root of $\Delta(\ld,s_*)=0$. This implies that
\be\label{crit}
\Delta(\ld_*,s_*)=0,\quad d_2\int_\bR{J_2(y)ye^{\lambda_*y}dy}=-d_2\int_\bR{J_2(y)ye^{-\lambda_*y}dy}=s_*.
\ee
Since $J_2$ is compactly supported, there exists $\tau>0$ such that $J_2(z)=0$ for $|z|\ge \tau$. Up to increasing $L$ without loss of generality, from now on we assume that $\hat{z}_1-z_1>\tau$, where $\Psi(z_1)=\Psi(\hat{z}_1)=b-1$.
Then, for any $z < z_1$ and $y \in [-\tau, \tau]$, we have $z-y<\hat{z}_1$ and so $\opsi(z-y)\le\Psi(z-y)$. It follows that
\beaa
&&\int_\bR{J_2(y)\opsi(z-y)dy}=\int_{-\tau}^\tau J_2(y)\opsi(z-y)dy\\
&\le& \int_{-\tau}^\tau {J_2(y)\Psi(z-y)dy}=\int_\bR{J_2(y)\{-L (z-y) e^{\lambda_*(z-y)}\}dy},\; \forall\, z<z_1.
\eeaa
With this estimate and using \eqref{crit}, we compute
\beaa
&&d_2\mathcal{N}_2[\opsi](z)-s_*\opsi'(z)+r_2\opsi(z)[-1+b\ophi(z)-\opsi(z)]\\
&\le& d_2\left[\int_\bR{J_2(y)\{-L (z-y) e^{\lambda_*(z-y)}\}dy}-(-Lz e^{\lambda_*z})\right]\\
&&\quad-s_*(-Le^{\lambda_*z}-\lambda_*Lze^{\lambda_*z})+r_2(-Lze^{\lambda_*z})(-1+b+Lze^{\lambda_*z}))\\
&=&-Lze^{\lambda_*z}\Big\{d_2\Big[\int_\bR{J_2(y)e^{-\lambda_*y}dy}-1\Big]-s_*\lambda_*+r_2(b-1)\Big\}\\
&&\quad+Le^{\lambda_*z}\Big(d_2\int_\bR{J_2(y)ye^{-\lambda_*y}dy}+s_* - r_2Lz^2e^{\lambda_*z}\Big)\\
&=&-r_2L^2z^2e^{2\lambda_*z}\le0
\eeaa
for $z<z_1$. Hence \eqref{u2} holds for all $z\neq z_1$.

{Next, \eqref{l1} can be obtained by taking $0<\mu<\min\{\rho,\lambda_*/2\}$ and $\eta$ large enough. The computation is the same as in the case $s > s_*$ and therefore we omit the details. For later use, we will also assume that $\eta$ is large enough so that
\beaa
z_2=-\ln(\eta)/\mu<-1/(\ld_*-\tl),\quad  -Lz_2 e^{(\lambda_*-\tl)z_2}<\eta ,
\eeaa
where $\tl\in(\lambda_*-\mu,\lambda_*)$. Thus, in particular, we have
\be\label{tleq}
-Lze^{\lambda_*z}< \eta e^{\tl z} < \eta e^{\mu z}\quad\mbox{ for $z\le z_2$.}
\ee

Lastly,} for \eqref{l2}, we set $z_3:=-(q/L)^2$. Then the function $(-Lz-q\sqrt{-z})e^{\lambda_*z}$ is positive for $z\in(-\infty,z_3)$ and has a unique maximal point in $(-\infty,z_3)$.
Note that $\upsi(z)$ can be written as
$$
\upsi(z)=\begin{cases}
(-Lz -q\sqrt{-z})e^{\ld_* z},\; z<z_3,\\
0,\; z\ge z_3.
\end{cases}
$$
With the above chosen constants $L,\mu,\tl,\eta$, we claim that \eqref{l2} holds for $z<z_3$ for a suitable large $q>L\sqrt{\ln(\eta)/\mu}$. Notice that this latest inequality implies that $z_3<z_2$.

Then, given $z<z_3$, it follows from~\eqref{tleq} and~$\upsi(z)\le -Lze^{\lambda_*z}$ that
\begin{align*}
r_2\upsi(z)[-1+b\uphi(z)-\upsi(z)]& =  r_2\upsi(z)[-1+b(1-\eta e^{\mu z}) -\upsi(z)]\\
&\ge r_2(b-1)\upsi(z)-r_2b\eta^2 e^{(\mu+\tl)z} - r_2\eta^2e^{2\tl z}.\\.
\end{align*}
Since $\upsi(z)\ge (-Lz-q\sqrt{-z})e^{\lambda_*z}$ in $\R$, using \eqref{crit} we get
\beaa
&&d_2\mathcal{N}_2[\upsi](z)-s_*\upsi'(z)+r_2 \upsi(z)[-1+b\uphi(z)-\upsi(z)]\\
&\ge& d_2\Big[\int_\bR{J_2(y)\{-L(z-y)-q\sqrt{-(z-y)}\}e^{\lambda_*(z-y)}dy}-(-Lz-q\sqrt{-z} ) e^{\lambda_* z} \Big]\\
&&\quad -s_*\Big[\Big(-L+\frac{q}{2\sqrt{-z}}\Big)e^{\lambda_*z}+\lambda_*(-Lz-q\sqrt{-z})e^{\lambda_* z}\Big]\\
&&\qquad +r_2(b-1)(-Lz-q\sqrt{-z})e^{\lambda_*z}-r_2 b \eta^2 e^{(\mu +\tl) z} - r_2 \eta^2 e^{2 \tl z} \\
&=&-Lze^{\lambda_*z}\Big\{d_2\Big[\int_\bR{J_2(y)e^{-\lambda_*y}dy}-1\Big]-s_*\lambda_*+r_2(b-1)\Big\}+ Le^{\lambda_*z}\Big(d_2\int_\bR{J_2(y)ye^{-\lambda_*y}dy}+s_*\Big)\\
&&\quad + e^{\lambda_*z}\Big\{d_2\Big[\int_\bR{J_2(y)(-q\sqrt{-(z-y)})e^{-\lambda_*y}dy}+q\sqrt{-z}\Big]+s_*\Big[q\lambda_*\sqrt{-z}-\frac{q}{2\sqrt{-z}}\Big]\Big\}\\
&&\qquad+ e^{\lambda_*z}\Big\{-r_2q(b-1)\sqrt{-z}-r_2b\eta^2 e^{(\mu+\tl-\lambda_*)z}-r_2\eta^2 e^{(2\tl-\lambda_*)z} \Big\}\\
&=&e^{\lambda_* z}[q I_1(z)-I_2(z)],
\eeaa
for $z < z_3$, where
\beaa
&&I_1(z):=-d_2\Big[\int_\bR{J_2(y)\sqrt{-(z-y)}e^{-\lambda_*y}dy}-\sqrt{-z}\Big]+s_*\Big(\lambda_*\sqrt{-z}-\frac{1}{2\sqrt{-z}}\Big)-r_2(b-1)\sqrt{-z},\\
&&I_2(z):=r_2b\eta^2 e^{(\mu+\tl-\lambda_*)z}+r_2\eta^2 e^{(2\tl-\lambda_*)z}.
\eeaa

To proceed further, from \eqref{crit} we first write
\beaa
I_1(z)&=&-d_2\int_\bR{J_2(y)(\sqrt{-(z-y)}-\sqrt{-z})e^{-\lambda_*y}dy}-\frac{s_*}{2\sqrt{-z}}\\
&=&-d_2\int_\bR{J_2(y)(\sqrt{-(z-y)}-\sqrt{-z})e^{-\lambda_*y}dy}+\frac{d_2}{2\sqrt{-z}}\int_\bR{J_2(y)ye^{-\lambda_*y}dy}\\
&=&d_2\Big[\int_\bR{J_2(y)\Big(-\sqrt{-(z-y)}+\sqrt{-z}+\frac{y}{2\sqrt{-z}}\Big)e^{-\lambda_* y}dy}\Big].
\eeaa
It follows from standard real analysis (see also the proof of~\cite[Theorem 3.4]{DGLP19}) that
$$-\sqrt{-(z-y)}+\sqrt{-z}+\frac{y}{2\sqrt{-z}}\ge \frac{y^2}{8(-z+\tau)^{3/2}}\quad\mbox{for $|y|<\tau$}.$$
Thus, we obtain
$$I_1(z)\ge \frac{d_2}{8(-z+\tau)^{3/2}}\int_\bR{J_2(y)y^2 e^{-\lambda_*y}dy}\quad\mbox{for $z<z_3$}.$$

Since $\mu+\tl>\lambda_*$ and $2\tl-\lambda_*>2(\lambda_*-\mu)-\lambda_*=\lambda_*-2\mu>0$, from the choices of $\mu$ and~$\tl$, the quantity
$$Q:=\frac{\max_{z<0}\{8(-z+\tau)^{3/2}I_2(z)\}}{d_2\int_\bR{J_2(y)y^2e^{-\lambda_*y}dy}}$$
is a well-defined finite number. Then, by choosing $q\ge \max\{Q,L\sqrt{\ln(\eta)/\mu}\}$, we have
$$d_2\mathcal{N}_2[\upsi](z)-s_*\upsi'(z)+r_2 \upsi(z)[-1+b\uphi(z)-\upsi(z)]\ge e^{\lambda_* z}[q I_1(z)-I_2(z)]\ge0$$
for $z<z_3$. Hence \eqref{l2} holds for all $z\neq z_3$.
We conclude that $(\ophi,\opsi)$ and $(\uphi,\upsi)$ are a pair of upper and lower solutions of \eqref{TWS}.\medskip

{Finally, {by applying Lemma~\ref{luslem},} we have proved the following:
\begin{lemma}\label{existence_mixed}
Suppose that $b>1$ and $s \geq s_*$. Then there exists a solution $(\phi,\psi)$ of~\eqref{TWS} such that $0\le\uphi\le\phi\le 1$ and $0\le\upsi\leq \psi\leq b-1$ in $\bR$, where $\uphi$ and $\upsi$ are defined by either \eqref{supsub_s>s*} or \eqref{supsub_s=s*0}-\eqref{supsub_s=s*}, depending on $s >s_*$ or $s = s_*$.
\end{lemma}}

\subsection{Limits of wave tails}

{In order to complete the proof of Theorem~\ref{th:forced}, it remains to check the asymptotics of the forced wave profile at $\pm \infty$.} We first claim that
\be\label{r-bc}
(\phi,\psi)(\infty)=(0,0),
\ee
for any nonnegative solution $(\phi,\psi)$ of \eqref{TWS}, including the solutions constructed in Lemmas~\ref{existence} and~\ref{existence_mixed}.

For contradiction, we assume that $\phi^+:=\limsup_{z\to \infty}\phi(z)>0$.
Then there is a maximal sequence $\{z_n\}$ of $\phi$ such that $z_n\to \infty$ and $\phi(z_n)\to \phi^+$ as $n\to\infty$.
It follows from the $\phi$-equation in \eqref{TWS} and $\alpha(-\infty)<0$ that
\beaa
0&=&\limsup_{n\to\infty}\{d_1\mathcal{N}_1[\phi](z_n)+r_1\phi(z_n)[\alpha(-z_n)-\phi(z_n)-a\psi(z_n)]\}\\
&\le&r_1\phi^+[\alpha(-\infty)-\phi^+ - a\liminf_{n\to\infty}\psi(z_n)]<0,
\eeaa
a contradiction. Here the following inequality was used:
\beaa
\limsup_{n\to\infty}\mathcal{N}_1[\phi](z_n)=\limsup_{n\to\infty}\left\{\int_\bR J_1(y)\phi(z_n-y)dy-\phi(z_n)\right\}\le 0,
\eeaa
since (cf. \cite{W})
\beaa
\limsup_{n\to\infty}\left\{\int_\bR J_1(y)\phi(z_n-y)dy\right\}\le\limsup_{z\to \infty}\left\{\int_\bR J_1(y)\phi(z-y)dy\right\}\le \limsup_{z\to \infty}\phi(z).
\eeaa
This proves that $\phi(\infty)=0$.

Similarly, we assume for contradiction that $\psi^+:=\limsup_{z\to \infty}\psi(z)>0$.
Then we have a maximal sequence $\{z_n\}$ of $\psi$ such that $z_n\to \infty$ and $\psi(z_n)\to \psi^+$ as $n\to\infty$.
It follows from the $\psi$-equation in \eqref{TWS} that
\beaa
0&=&\limsup_{n\to\infty}\{d_2\mathcal{N}_2[\psi](z_n)+r_2\psi(z_n)[-1+b\phi(z_n)-\psi(z_n)]\}\\
&\le&r_2\psi^+(-1 - \psi^+)<0,
\eeaa
a contradiction again. Hence $\psi(\infty)=0$ and so we have proved \eqref{r-bc}.\medskip

Next, we claim that {the solution obtained from Lemma~\ref{existence} satisfies $(\phi,\psi)(-\infty)=(u_*,v_*)$. Notice that, since $\phi\ge\uphi$ and $\psi\ge\upsi$, we have
\be\label{phi-p}
\phi^-:=\liminf_{z\to -\infty}\phi(z)\ge \gamma_1,\;\psi^-:=\liminf_{z\to-\infty}\psi(z)\ge \gamma_2,
\ee
where}
\beaa
\gamma_1:=1-a(b-1)>0,\quad \gamma_2:=-1+b[1-a(b-1)]=(b-1)(1-ab)>0.
\eeaa
With \eqref{phi-p} in hand, the following result can be proved by a similar argument as that of \cite{cgy17} {with some modifications.}

\begin{lemma}\label{uvstar}
It holds that $(\phi,\psi)(-\infty)=(u_*,v_*)$ for the solution $(\phi,\psi)$ of \eqref{TWS} obtained from Lemma~\ref{existence}.
\end{lemma}

\begin{proof}
Consider the following functions
\begin{align*}
&m_1(\theta):=\theta u_*+(1-\theta)(\gamma_1-\ep),\; M_1(\theta):=\theta u_*+(1-\theta)(1+\ep),\quad\theta\in[0,1],\\
&m_2(\theta):=\theta v_*+(1-\theta)(\gamma_2-k_1\ep),\; M_2(\theta):=\theta v_*+(1-\theta)(b-1+k_2\ep),\quad\theta\in[0,1],
\end{align*}
where {$k_1:={2}/{a}$}, $k_2:=(b+1/a)/2$ and $\ep$ satisfies
\be\label{ep-p}
0<\ep<\min\left\{\gamma_1,\frac{a \gamma_2}{2}\right\}.
\ee
Note that {$k_1 > 2b$ and} $k_2\in(b,1/a)$, where $1<b<1/a$ due to \eqref{ab2}.

By \eqref{phi-p} and Lemma~\ref{existence}, it is obvious that
\be\label{ineq1}
m_1(\theta)<\phim\le\phip<M_1(\theta),\;\;m_2(\theta)<\psim\le\psip<M_2(\theta)
\ee
holds for $\theta=0$.
Hence the quantity
$$\theta_0:=\sup\{\theta\in[0,1):(\ref{ineq1})\mbox{ holds}\} \in (0,1].$$ is well-defined.
Since $0<\gamma_1<u_*<1$ and $0<\gamma_2<v_*<b-1$, the function $m_i(\theta)$ (resp. $M_i(\theta)$) is increasing {(resp. decreasing)} in $\theta\in[0,1]$, $i=1,2$.
Moreover, $m_1(1)=M_1(1)=u_*$ and $m_2(1)=M_2(1)=v_*$. Hence the lemma follows if we can show that $\theta_0=1$.

For contradiction, we suppose that $\theta_0<1$. Then, by passing {to the limit as $\theta \to \theta_0$ in~\eqref{ineq1}}, we obtain
$$m_1(\theta_0)\le\phim\le\phip\le M_1(\theta_0),\;\;m_2(\theta_0)\le\psim\le\psip\le M_2(\theta_0).$$
By the definition of $\theta_0$ and the continuity of $m_i(\theta)$ and $M_i(\theta)$, $i=1,2$, inequalities~\eqref{ineq1} cannot hold for $\theta=\theta_0$.
This means that at least one of the following equalities holds:
\be\label{cases}
\phim=m_1(\theta_0),\;\;\phip=M_1(\theta_0),\;\;\psim=m_2(\theta_0),\;\;\psip=M_2(\theta_0).
\ee

First, we assume that $\phim=m_1(\theta_0)$. If $\phi$ is eventually monotone, then $\phi(-\infty)$ exists. Furthermore, $\liminf_{z\to-\infty}\phi'(z)=0$ or $\limsup_{z\to-\infty}\phi'(z)=0$.
Then we can find a sequence $\{z_n\}$ with $z_n\to-\infty$ as $n\to\infty$ such that $\lim_{n\to\infty}\phi'(z_n)=0$ and $\lim_{n\to\infty}\phi(z_n)=m_1(\theta_0)$.
Since $\limsup_{n\to\infty}\psi(z_n)\le M_2(\theta_0)$, we have
\begin{align*}
&\liminf_{n\rightarrow\infty}[\alpha(-z_n)-\phi(z_n)-a\psi(z_n)]\\&\ge 1- [\theta_0u_*+(1-\theta_0)(\gamma_1-\ep)] -a[\theta_0v_*+(1-\theta_0)(b-1+k_2\ep)]\\
&=\ep(1-ak_2)(1-\theta_0)>0.
\end{align*}
The last inequality holds by the choice of {$k_2$ and $\ep$}. On the other hand, by Fatou's lemma we have
\begin{align*}
\liminf_{n\to\infty}\mathcal{N}_1[\phi](z_n)
&=\liminf_{n\rightarrow\infty}\left[\int_\bR{J_1(y)\phi(z_n-y)dy}-\phi(z_n)\right]\\&\ge \liminf_{n\rightarrow\infty}\int_\bR{J_1(y)\phi(z_n-y)dy} + \liminf_{n\rightarrow\infty}(-\phi(z_n))\\
&\ge \int_\bR{m_1(\theta_0)J_1(y)}-m_1(\theta_0)=0.
\end{align*}
Then, from the first equation of \eqref{TWS}, we obtain
\begin{align*}
0=\liminf_{n\to\infty}s\phi'(z_n)&\ge\liminf_{n\to\infty}d_1\mathcal{N}_1[\phi](z_n)+\liminf_{n\to\infty}\{r_1\phi(z_n)[\alpha(-z_n)-\phi(z_n)-a\psi(z_n)]\}>0,
\end{align*}
which is a contradiction.

Next, we assume that $\phi$ is oscillatory at $-\infty$. Then, we can choose a sequence $\{z_n\}$ of local minimum points of $\phi$ with $z_n\to-\infty$ as $n\to\infty$ such that $\lim_{n\to\infty}\phi(z_n)=m(\theta_0)$.
In particular, $\phi'(z_n)=0$ for all $n$. Similarly as above, we reach the same contradiction
$$0=\liminf_{n\to\infty}s\phi'(z_n)\ge\liminf_{n\to\infty}d_1\mathcal{N}_1[\phi](z_n)+ \liminf_{n\to\infty}\{r_1\phi(z_n)[\alpha(-z_n)-\phi(z_n)-a\psi(z_n)]\}>0 .$$
Hence $\phim=m(\theta_0)$ cannot happen.

The other cases in (\ref{cases}) can be treated similarly, using
\beaa
\limsup_{n\to\infty}\int_\bR{J_i(y)f(z_n-y)dy}\leq \int_\bR J_i(y)\{\limsup_{n\to\infty}f(z_n-y)\}dy
\eeaa
for any bounded continuous function $f$ in $\bR$ and the following inequalities:

i) for $\phi^+=M_1(\theta_0)$,
\begin{align*}
&\limsup_{n\rightarrow\infty}[\alpha(-z_n)-\phi(z_n)-a\psi(z_n)]\\&\le 1- [\theta_0u_*+(1-\theta_0)(1+\ep)] -a[\theta_0v_*+(1-\theta_0)(\gamma_2-k_1\ep)]\\
&=(1-\theta_0)[(ak_1-1)\ep- a\gamma_2]<0;
\end{align*}

ii) for $\psi^-=m_2(\theta_0)$,
\begin{align*}
&\liminf_{n\rightarrow\infty}[-1+b\phi(z_n) - \psi(z_n)]\\&\ge -1+ b [\theta_0u_*+(1-\theta_0)(\gamma_1-\ep)] -[\theta_0v_*+(1-\theta_0)(\gamma_2-k_1\ep)]\\
&=\ep(k_1-b)(1-\theta_0)>0;
\end{align*}

iii) for $\psi^+=M_2(\theta_0)$,
\begin{align*}
&\limsup_{n\rightarrow\infty}[-1+b\phi(z_n) - \psi(z_n)]\\&\le -1+ b [\theta_0u_*+(1-\theta_0)(1+\ep)] -[\theta_0v_*+(1-\theta_0)(b-1+k_2\ep)]\\
&=\ep(b-k_2)(1-\theta_0)<0.
\end{align*}
\noindent
The lemma is proved.
\end{proof}

Therefore, Theorem~\ref{th:forced} follows from Lemma~\ref{existence}, \eqref{r-bc} and Lemma~\ref{uvstar}. Moreover, for the mixed front-pulse type waves, from the constructed upper-lower solutions it is clear that $(\phi,\psi)(-\infty)=(1,0)$; see \eqref{supsub_s>s*} and \eqref{supsub_s=s*0}-\eqref{supsub_s=s*}. Also, by the strong maximum principle, $\psi(z)>0$ for all $z\in\bR$, since $\psi\ge\upsi\ge 0$ and $\upsi\not\equiv 0$.
Hence the existence part of Theorem~\ref{th:forced2} follows from Lemma~\ref{existence_mixed} and \eqref{r-bc}.

\subsection{Non-existence of mixed front-pulse type forced waves}

For the non-existence of mixed type forced waves, we actually have the following result.
\begin{theorem}
Assume that $s < s_*$. Then \eqref{TWS} does not have any positive solution satisfying $(\phi,\psi)(-\infty)=(1,0)$.
\end{theorem}

\begin{proof}
For contradiction, suppose that there exists a positive solution $(\phi,\psi)$ of \eqref{TWS} such that $(\phi,\psi)(-\infty)=(1,0)$.
Let $\zeta(z)=\psi'(z)/\psi(z)$. From the second equation of \eqref{TWS}, we obtain
\be\label{zeta}
s\zeta(z)=d_2\Big[\int_\bR{J_2(y)e^{\int_{z}^{z-y}{\zeta(\tau)d\tau}}dy}-1\Big]+r_2[-1+b\phi(z)-\psi(z)].
\ee
Then it follows from \cite[Proposition 3.7]{zlw12} that the limit
$\zeta:=\lim_{z\to-\infty}\zeta(z)$ exists and solves
$$s\zeta = d_2\Big[\int_\bR{J_2(y)e^{\zeta y}dy}-1\Big] + r_2(b-1),$$
which means $\Delta(\zeta,s)=0$. However, this contradicts the fact that $\Delta(\lambda,s)>0$ for any $\lambda>0$ when $s < s_*$.
Hence the proof is done.
\end{proof}

We thereby conclude the proof of Theorem~\ref{th:forced2}.

\section{Spreading dynamics: the nonlocal case}
\setcounter{equation}{0}

{In this section, we shall give some results on the persistence and extinction of both species. The proofs are inspired by \cite{lwz18,wwz19}. Recall that we defined}
\be\label{speed}
s^*:=\inf_{\lambda>0}\frac{d_1(\int_\bR{J_1(y)e^{\lambda y}dy}-1)+r_1}{\lambda},\;
s_{*}:=\inf_{\lambda>0}\frac{d_2(\int_\bR{J_2(y)e^{\lambda y}dy}-1)+r_2(b-1)}{\lambda}.
\ee
In the sequel, we let $(u,v)$ be a solution of \eqref{pp} and \eqref{pp-ic} with initial data $(u_0,v_0)\in X_1\times X_{b-1}$, where
\beaa
X_K := \{\varphi \in C^0(\bR) : 0\le \varphi \le K \mbox{ for $x\in\bR$}\}
\eeaa
for a positive constant $K$. Also, let $s>0$ be a given fixed constant.\medskip

\subsection{Extinction}
We first show that any species goes {to extinction uniformly in space as $t\to+\infty$ if it} cannot keep pace with the climate changing speed {even under the most favorable conditions
(i.e., absence of predators for the prey, and abundance of prey for the predator).}

\begin{theorem}\label{sp1}
Assume that {$(u_0,v_0) \in X_1 \times X_{b-1}$ and that} both $u_0$ and $v_0$ have nonempty compact supports. Then
\bea
&&\lim_{t\to\infty}u(x,t)=0 \mbox{ uniformly for $x\in\bR$, if $s>s^*$};\label{u-0}\\
&&\lim_{t\to\infty}v(x,t)=0 \mbox{ uniformly for $x\in\bR$, if $s>s_*$.}\label{v-0}
\eea
\end{theorem}

In particular, Theorem~\ref{sp1} includes the first item of Theorem~\ref{spread_fastprey}, and the extinction of~$u$ in the first item of Theorem~\ref{spread_fastpred}. Notice that it also implies that $v$ is driven to extinction when $s= s^* > s_*$, though for the sake of conciseness we omitted it from our main results.

\begin{proof}
{First assume that $s > s^*$. Let $\delta>0$ be arbitrarily small and $z$ be the solution of the initial value problem
\begin{equation}\label{scalar}
\begin{cases}
z_t(x,t)=d_1\mathcal{N}_1[z](x,t)+r_1z(x,t)[{\alpha(x-st)} + \delta -z(x,t)], &\; x\in\bR,\, t>0,\\
z(x,0)=u_0(x)\geq0, &\; x\in\bR.
\end{cases}
\end{equation}From \cite[Theorem 3.1]{lwz18} and $u_0 \leq 1 < 1+\delta$, we have that $\lim_{t\to\infty}z(x,t)=0$ uniformly for $x\in\bR$, if $s>s^* (\delta)$, where
$$ s^* (\delta ):= \inf_{\lambda>0}\frac{d_1 (\int_\bR{J_1(y)e^{\lambda y}dy}-1)+r_1(1 +\delta)}{\lambda} \to s^*\;\mbox{as $\delta \to 0$.}$$
For the sake of completeness, let us point out that this result of~\cite[Theorem 3.1]{lwz18} follows first from a comparison with the solution of the homogeneous equation with $\alpha \equiv \alpha (\infty)$,
which insures that $z$ converges to 0 as $t \to +\infty$ uniformly on the set $\{x \geq \frac{s^* (\delta) + s}{2} t\}$,
and then the uniform convergence on the whole space comes from the fact that $\alpha (x -st) + \delta$ becomes negative on $\{x \leq \frac{s^* (\delta) + s}{2} t\}$.}

{Then, by the comparison principle for the scalar equation, we get that $u(x,t)\leq z(x,t)$ for $x\in\bR,\, t>0$, and \eqref{u-0} follows.}

{Next, consider the case when $s > s_*$. We again pick $\delta>0$ arbitrarily small, and introduce a perturbed problem
\begin{equation}\label{scalar_delta}
\begin{cases}
z_t(x,t)=d_1\mathcal{N}_1[z](x,t)+r_1z(x,t)[{\alpha(x-st)} + \delta -z(x,t)], &\; x\in\bR,\, t>0,\\
z(x,0)=u_0(x)\geq0, &\; x\in\bR.
\end{cases}
\end{equation}
From \cite[Theorem 4.5]{lwz18}, problem~\eqref{scalar_delta} has a forced wave solution $\psi_\delta(x-st)$ such that $\psi_\delta$  is nondecreasing, $\psi_\delta(-\infty)=0$
and $\psi_\delta(\infty)=\alpha(\infty) + \delta =1+ \delta$.}
Since $u_0$ is compactly supported with {$\max_{x\in\bR}u_0(x)\leq1$}, we can choose $x_0>0$ such that $u_0(x)<\psi_\delta(x+x_0)$ for $x\in\bR$.
{Then, since $\alpha$ is nondecreasing, we can easily check that $\overline{u}(x,t):=\psi_\delta(x-st+x_0)$ satisfies
$$\overline{u}_t(x,t)\ge d_1\mathcal{N}_1[\overline{u}](x,t) + r_1\overline{u}(x,t)[\alpha(x-st)-\overline{u}](x,t).$$
Therefore, by comparison, $u(x,t)\leq \psi_\delta(x-st+x_0)$ for $x\in\bR,\, t>0$.} 

Now, let $w$ be the solution of the initial value problem
\begin{equation}\label{scalarv}
\begin{cases}
w_t(x,t)=d_2\mathcal{N}_2[w](x,t)+r_2w(x,t)[b\psi_\delta (x-st+x_0)-1-w(x,t)], &\; x\in\bR,\, t>0,\\
w(x,0)=v_0(x)\geq0, &\; x\in\bR.
\end{cases}
\end{equation}
From the comparison principle, $v(x,t)\le w(x,t)$ for $x\in \bR,\, t>0$. Suppose that $s>s_*$, {so that we can choose $\delta>0$ small enough such that
$$s> s_* (\delta):= \inf_{\lambda>0}\frac{d_2(\int_\bR{J_2(y)e^{\lambda y}dy}-1)+r_2(b + b\delta-1)}{\lambda}.$$
Since $\max_{x\in\bR}v_0(x) \leq b-1$ and $h(x-st):=b\psi_\delta (x-st+x_0)-1$ is nondecreasing with $h(-\infty)=-1$ and $h(\infty)=b +b \delta -1>0$, another use of}
\cite[Theorem 3.1]{lwz18} asserts that $w(x,t)$ converges to $0$ uniformly for $x\in\bR$ as $t\to\infty$.
Hence \eqref{v-0} follows and the theorem is proved.
\end{proof}

Next, {we consider the case $s^*\leq s_*$, in which the predator may (theoretically) keep pace with a faster climate change than the prey. Then actually} we still have the extinction for both species when $s>s^*$,
due to the shortage of the prey, regardless of the {value} of $s_*$.

\begin{theorem}\label{sp2}
Suppose that {$s^*\leq s_*$}. Assume that {$(u_0,v_0) \in X_1 \times X_{b-1}$ and that} both $u_0$ and $v_0$ have nonempty compact supports.
If $s>s^*$, then
$$\lim_{t\to\infty}v(x,t)=0\, \mbox{ uniformly for $x\in\bR$.}$$
\end{theorem}

{One can check that Theorem~\ref{sp2}, together with Theorem~\ref{sp1}, completes the proof of the first item in Theorem~\ref{spread_fastpred}.}

\begin{proof}
{Since $s>s^*$, by \eqref{u-0}, for any small $\ep\in(0,1/b)$ there exists $T_0>0$ such that $u(x,t)\leq \ep$ for all $x\in\bR$ for $t\ge T_0$. 
Set
$$\overline{w}(x,t):=(b-1)e^{-\sigma (t-T_0)},$$
where
$$0 < \sigma <r_2 (1 - b \varepsilon).$$
Then a straightforward computation gives
\beaa
&&\overline{w}_t- \{d_2\mathcal{N}_2[\overline{w}]+r_2\overline{w}(b\ep -1 -\overline{w})\}\\
&=&-(b-1)e^{-\sigma(t-T_0)}[\sigma - r_2(1-b\ep) -r_2(b-1)e^{-\sigma(t-T_0)}]\ge 0.
\eeaa
Since $v(x,T_0)\le b-1=\overline{w}(x,T_0)$ for $x\in\bR$, by comparison, $v(x,t)\le \overline{w}(x,t)$ for $x\in\bR,\,t\ge T_0$.
Hence $v(x,t)\to 0$ as $t\to+\infty$ uniformly for $x\in\bR$. This completes the proof of the theorem.}
\end{proof}

{Therefore, according to Theorems~\ref{sp1} and \ref{sp2}, the only chances for either species to survive are when either $s<s^*$ or $s<s_*$.
Even then, we cannot expect the species to spread uniformly when the initial condition is localized. Indeed, they may neither persist behind the shifting climate,
nor invade the favorable part of the environment faster than their `optimal' speeds. The following theorem gives the vanishing of each species in these outer cone regions,
and in particular it includes the second limit of the second item, the first and second limits of the third item of Theorem~\ref{spread_fastprey}, as well as the second item of Theorem~\ref{spread_fastpred}.}

\begin{theorem}\label{sp4} 
{Assume that $(u_0,v_0) \in X_1 \times X_{b-1}$.}
The following statements hold.

{\rm (i)} 
Suppose that $u_0(x)=v_0(x)=0$ for $x\le K_1$ for some constant $K_1$.
Then
\beaa
\lim_{t\to\infty}\sup_{x\le (s-\zeta)t}u(x,t)=0,\;\mbox{if {$s\leq s^*$},}\quad
\lim_{t\to\infty}\sup_{x\le (s-\zeta)t}v(x,t)=0,\;\mbox{if {$s\leq s_*$},}
\eeaa
for any small $\zeta>0$.

{\rm (ii)} If $u_0(x)=v_0(x)=0$ for $x\ge K_2$ for some constant $K_2$, then
$$\lim_{t\to\infty}\sup_{x\ge (s^*+\tau)t}u(x,t)=0,\;\;\lim_{t\to\infty}\sup_{x\ge (s_*+\tau)t}v(x,t)=0, \; \;{\lim_{t\to\infty}\sup_{x\ge (s^*+\tau)t}v(x,t)=0},$$
for any $\tau>0$.
\end{theorem}
\begin{proof}
The case for $u$ in part (i) follows easily from a comparison with the solution~$z$ of~\eqref{scalar} {with $\delta>0$ arbitrarily small.
Indeed, according to~\cite[Theorem 3.3 (i)]{lwz18}, using $s\leq s^*< s^* (\delta)$ we have that $z$ converges to 0 as $t\to +\infty$ uniformly with respect to $x \leq (s - \zeta) t$, for any small $\zeta >0$.
This simply follows from the fact that $z$ is below some shift of the forced wave solution of~\eqref{scalar}.}

The case for $v$ is similar. {As in the proof of Theorem~\ref{sp1}, we have that $v$ is a subsolution of \eqref{scalarv}, also with arbitrarily small~$\delta >0$, for some $x_0 >0$.
Then~\cite[Theorem 3.3 (i)]{lwz18} again gives the wanted conclusion}.

For part (ii), let $\tau>0$ be given and let $\lambda_1>0$ be {a positive solution} of
\be\label{ld1}
d_1\left(\int_\bR{J_1(y)e^{\lambda y}dy-}1\right)+r_1=\lambda(s^*+\tau/2).
\ee
Then $\overline{u}(x,t):=A e^{-\lambda_1[x-(s^*+\tau/2)t]}$ is a solution of the following linear equation
$$\overline{u}_t(x,t)=d_1\mathcal{N}_1[\overline{u}](x,t) + r_1\overline{u}(x,t),$$
for any positive constant $A$.

Since $u_0(x)=0$ for $x\ge K_2$ and $u_0\le 1$, we can choose a positive constant $A$ large enough such that $u_0(x)\le A e^{-\lambda_1 x}$ for all $x\in\bR$.
Then, by the comparison principle,
$$0\le u(x,t)\le \overline{u}(x,t)=A e^{-\lambda_1[x-(s^*+\tau)t]} e^{-\lambda_1 \tau t/2},$$
for $x\in\bR$, $t\ge 0$.
In particular, we have $0\le u(x,t)\le A e^{-\lambda_1 \tau t/2}$ for $x\ge (s^*+\tau)t$ and $t\ge0$.
This implies that $\lim_{t\to\infty}\sup_{x\ge (s^*+\tau)t}u(x,t)=0$.

Next, we compare $v$ with the function $\overline{v}:= B e^{-\lambda_2[x-(s_*+\tau/2)t]}$, where $\lambda_2$ is {the smaller positive solution} of
\be\label{ld2}
d_2\left(\int_\bR{J_2(y)e^{\lambda y}dy-}1\right)+r_2(b-1)=\lambda(s_*+\tau/2).
\ee
Then we obtain the {first} desired result for $v$ in (ii).

Lastly, recall from the above that $u \leq \overline{u}$ and also $u\le 1$. Therefore, for any $\tau >0$ there exist $A >0$ such that $v$ is a subsolution of
\begin{equation}\label{eq:pred_noprey1}
V_t(x,t)=d_{2}\mathcal{N}_2[V](x,t)+r_{2}V(x,t)[ b\min\{1,A e^{-\lambda_1 [x - (s^* + \tau/2) t]}\}  - 1 -V(x,t)],
\end{equation}
where $\lambda_1$ is a positive solution of \eqref{ld1}.
{Since the case for $s_*\le s^*$ is already included in the first result for $v$ in (ii), we only consider the case when $s_*>s^*$.}
{Then one can find $\lambda'>0$ and $B'>0$ large enough such that
$$\widehat{v}(x,t) :=  \min \left\{  b-1, \, B' e^{- \lambda' [x - (s^* + \tau/2) t]} \right\}$$
is a supersolution of~\eqref{eq:pred_noprey1}.}

Indeed, notice that
$$\lim_{\lambda \to 0} d_2 \left[ \int_{\R} J_2 (y) e^{\lambda y} dy -1 \right] = 0,$$
by (J1)-(J2). Thus we can choose $\lambda ' >0$ small enough such that
$$d_2 \left[ \int_{\R} J_2 (y) e^{\lambda' y} dy -1 \right] - \lambda' \left( s^* + \frac{\tau}{2} \right) < \frac{r_2}{2}.$$
Now, for those points $(x,t)$ with
\beaa
B' e^{- \lambda' [x - (s^* + \tau/2) t]}< b-1,\;\mbox{ or, }\;
e^{-\lambda_1 [x - (s^* + \tau/2) t]}=\left[ e^{- \lambda' [x - (s^* + \tau/2) t]}\right]^{\lambda_1/\lambda'}<\left[\frac{b-1}{B'}\right]^{\lambda_1/\lambda'},
\eeaa
the function $w(x,t):=B' e^{- \lambda' [x - (s^* + \tau/2) t]}$ satisfies
\bea
&&w_t-d_2\mathcal{N}_2[w]-r_{2}w[ b A e^{-\lambda_1 [x - (s^* + \tau/2) t]}  - 1 -w] \nonumber\\
&\geq&w\left\{\lambda'\left(s^*+\frac{\tau}{2}\right)-d_2\left[\int_\bR J_2(y)e^{\lambda' y}dy-1\right]-r_2\left[bA\left(\frac{b-1}{B'}\right)^{\lambda_1/\lambda'}-1\right]\right\} \nonumber \\
& \geq & w r_2 \left\{  \frac{1}{2} - bA\left(\frac{b-1}{B'}\right)^{\lambda_1/\lambda'} \right\} \nonumber \\
& \geq & 0, \nonumber
\eea
where the last inequality holds provided that $B'$ is large enough. It is straightforward to handle the case when $\widehat{v} = b-1$ and thus $\widehat{v}$ is a supersolution of~\eqref{eq:pred_noprey1}.
We obtain the second desired result for $v$ in (ii) by the comparison principle, and the theorem is thereby proved.
\end{proof}

\begin{remark}\label{rk1}
From part (i) of Theorem~\ref{sp4} and Theorem~\ref{sp1}, we see that $u(x,t)\to 0$ and $v(x,t)\to 0$ as $t\to\infty$ for each $x\in\bR$.
 Biologically, this means that no species can survive in the long run without adjusting its habitat.
\end{remark}


\subsection{{Survival}}
{We finally turn to a scenario where a species survives. We start with the simplest case where the prey is able to spread away from the predator.
More precisely,} we give a result on the saturation of the prey in a cone region without predator when $s^*>s_*$,
{which corresponds to the third limits in the second and third items of Theorem~\ref{spread_fastprey}.}

\begin{theorem}\label{sp3} {Let $(u_0,v_0)\in X_1\times X_{b-1}$. Assume that $u_0(x)>0$ on a closed interval and~$v_0$ has a nonempty compact support.
\begin{itemize}
\item If $s_* < s < s^*$, then
\beaa
\lim_{t\to\infty}\left\{\sup_{(s+\ep)t\le x\le (s^*-\ep)t}|u(x,t)-1|\right\}=0
\eeaa
for any $\ep\in(0,(s^*-s)/2)$.
\item If $s \leq s_* < s^*$, then
\beaa
\lim_{t\to\infty}\left\{\sup_{(s_*+\ep)t\le x\le (s^*-\ep)t}|u(x,t)-1|\right\}=0
\eeaa
for any $\ep\in(0,(s^*-s_*)/2)$.
\end{itemize}}
\end{theorem}
\begin{proof}
First, {we consider the first item and assume that $s_* < s < s^*$.} Recall from \eqref{v-0} that, since $s > s_*$, we have $v(x,t)\rightarrow 0$ uniformly for $x\in\bR$ as~$t\to\infty$.
Next, for any $\delta\in(0,1)$, there exists $T_1>0$ such that $v(x,t)\leq \delta/a$ for all $x\in\bR$ for $t\ge T_1$.
Let $\underline{u}$ and $\overline{u}$ be the solutions of
\begin{equation}\label{loweru}
\begin{cases}
\underline{u}(x,t)=d_1\mathcal{N}_1[\underline{u}](x,t)+r_1\underline{u}(x,t)[\alpha(x-st)-\delta-\underline{u}(x,t)], &\; x\in\bR,\, t>T_1,\\
\underline{u}(x,0)=u(x,T_1), &\; x\in\bR,
\end{cases}
\end{equation}
and
\begin{equation}\label{upperu}
\begin{cases}
\overline{u}_t(x,t)=d_1\mathcal{N}_1[\overline{u}](x,t)+r_1\overline{u}(x,t)[\alpha(x-st)-\overline{u}(x,t)], &\; x\in\bR,\, t>T_1,\\
\overline{u}(x,0)=u(x,T_1), &\; x\in\bR,
\end{cases}
\end{equation}
respectively.
Then, by comparison, $\underline{u}(x,t)\le u(x,t)\leq \overline{u}(x,t)$ for $x\in\bR,\,t\ge T_1$.

Now fix $\ep\in(0,(s^*-s)/2)$. Then, from~\cite[Theorem 3.3 (iii)]{lwz18}, {we have
\be\label{upper}
\lim_{t\to\infty}\left\{\sup_{(s+\ep)t\le x\le (s^*-\ep)t}|\overline{u}(x,t)-1|\right\}=0.
\ee
In other words, when there is no predator and if it is fast enough, the prey manages to spread ahead of the climate change.
This is a consequence of the existence of compactly supported subsolutions, moving with any speed less than but arbitrarily close to $s^*$,
in the homogeneous scalar equation for $u$ obtained by setting $\alpha \equiv 1$ and $v \equiv 0$.}

Now, let
$${s^*(- \delta)}:=\inf_{\lambda>0}\frac{d_1(\int_\bR{J_1(y)e^{\lambda y}dy}-1)+r_1[\alpha(\infty)-\delta]}{\lambda}.$$
Note that $s^*(-\delta)\uparrow s^*$ as $\delta\downarrow 0^+$. Hence $s^*(-\delta)>s$ and $\ep<[s^*(-\delta)-s]/2$, if $\delta\ll 1$.
Then, for any such small $\delta$, it follows from \cite[Theorem 3.3 (iii)]{lwz18} {again} that
\be\label{lower-d}
\lim_{t\to\infty}\left\{\sup_{(s+\ep /2)t\le x\le (s^*(-\delta)-\ep/2)t}|\underline{u}(x,t)-1+\delta|\right\}=0.
\ee
{Now choose $\delta_0 >0$ small enough such that $s^* - \varepsilon \leq s^* (-\delta_0) - \ep/2$.
Then for any $\delta \in (0,\delta_0)$ we put \eqref{upper} and \eqref{lower-d} together with $\underline{u} \leq u \leq \overline{u}$, and we find that
$$\limsup_{t\to \infty} \left\{ \sup_{(s+\varepsilon) t\leq x \leq (s^* - \varepsilon) t} | u (x,t)- 1| \right\} \leq \delta.$$
Since $\delta>0$ is arbitrarily small, this completes the proof of the first item of Theorem~\ref{sp3}.}

%

Now we turn to the second item. According to part (ii) of Theorem~\ref{sp4}, we know that
$$\lim_{t \to \infty} \sup_{x \geq (s_* + \tau) t} v(x,t)=0,$$
for $\tau >0$ arbitrarily small. This means that in any moving frame with speed larger than $s_* + \tau$,
we can use the same construction of compactly supported subsolutions as in the homogeneous case or, e.g.~\cite{lwz18},
to reach the conclusion that $u(x,t)$ converges to $1$ as $t \to \infty$ in the moving frames with speeds in the interval $(s_* + \tau, s^* - \tau)$, when $s\leq s_*<s^*$.
For the sake of conciseness, here we show how to obtain this result by a comparison with a scalar equation with shifting heterogeneity as follows.

{Indeed, we also know from the proof of part (ii) of Theorem~\ref{sp4} that
\beaa
v (x,t )\leq  \min\{ b-1, B e^{-\lambda_2 [ x - (s_* + \tau/2) t]}\},
\eeaa
for any $\tau >0$ arbitrarily small, and for some $B, \lambda_2 >0$. It follows that
\beaa
{\alpha (x - st)  -a v (x,t) \geq \widehat{\alpha} (x -  (s_* + \tau/2) t) }
\eeaa
for some function $\widehat{\alpha}$ satisfies $(\alpha 1)$ and $(\alpha 2)$, {using $s\le s_*$}.
In particular,
$$u_t (x,t ) \geq d_{1}\bNu(x,t)+r_{1}u(x,t)[{\widehat{\alpha} (x-(s_* + \tau/2)t)}-u(x,t)], \; x\in\bR,\, t>0 .$$
From \cite[Theorem 3.3 (iii)]{lwz18} and the comparison principle, one gets that
$$\liminf_{ t \to \infty}   \inf_{ (s_* + \tau) t \leq x \leq (s^* - \tau) t} u (x,t) \geq 1.$$
{Since $u\le 1$,} this ends the proof of the second item of Theorem~\ref{sp3}.}
\end{proof}


Finally, it remains to study the dynamics of $u$ and $v$ when both $s^*$ and $s_*$ are bigger than the changing speed $s$.
First, we give a proof of Theorem~\ref{upos}, which shortly follows from known results for the scalar equation with nonlocal dispersal climate change~\cite{lwz18}.

\begin{proof}[Proof of Theorem~\ref{upos}]
 Since $v\le b-1$, comparing $u$ with the solution $\underline{u}$ of the initial value problem
\begin{equation}\label{upperuu}
\begin{cases}
\underline{u}_t(x,t)=d_1\mathcal{N}_1[\underline{u}](x,t)+r_1\underline{u}(x,t)[\alpha(x-st)-a(b-1)-\underline{u}(x,t)], &\; x\in\bR,\, t>0,\\
\underline{u}(x,0)=u_0(x) &\; x\in\bR,
\end{cases}
\end{equation}
we obtain $u(x,t)\ge \underline{u}(x,t)$ for all $x\in\bR$, $t>0$.
{Then, applying \cite[Theorem 3.3 (iii)]{lwz18}, we deduce that
\be\label{ubar}
\liminf_{t\to\infty}\{\inf_{(s+\ep)t\le x\le (s^{**}-\ep)t}u(x,t)\}\ge \lim_{t\to\infty}\{\inf_{(s+\ep)t\le x\le (s^{**}-\ep)t}\underline{u}(x,t)\}=1-a(b-1)>0
\ee
for any $\ep\in(0,(s^{**}-s)/2)$.
Hence the proof is done.}
\end{proof}

With Theorem~\ref{upos} in hand, we can now turn to the proof of Theorem~\ref{vpos}.

\begin{proof}[Proof of Theorem~\ref{vpos}]
Let $\underline{v}$ be the solution of the following initial value problem
\begin{equation}\label{lowerv}
\begin{cases}
\underline{v}_t=d_2\mathcal{N}_2[\underline{v}](x,t)+r_2\underline{v}(x,t)[h(x,t)-\underline{v}(x,t)], &\; x\in\bR,\, t>0,\\
\underline{v}(x,0)=v_0(x), &\; x\in\bR,
\end{cases}
\end{equation}
where $h(x,t):=-1+b\underline{u} (x,t)$.
Then, by comparison, $v (x,t) \ge \underline{v} (x,t)$ for all $x\in\bR$, $t>0$.

Let $\lambda^* >0 $ be such that $s_{**}=\chi(\lambda^*)$, where
\beaa
\chi(\lambda):=\frac{d_2(\int_\bR{J_2(y)e^{\lambda y}dy}-1)+r_2(b-1)(1-ab)}{\lambda}.
\eeaa
Following \cite{lwz18}, we also introduce
\bea
&&\varphi(\lambda,\gamma)=d_2\int_\bR {J_2(y)e^{\lambda y}\frac{\sin(\gamma y)}{\gamma}dy}=d_2\int_0^\infty J_2(y)\left(e^{\lambda y}-e^{-\lambda y}\right)\frac{\sin(\gamma y)}{\gamma}dy,\label{vphi}\\
&&{\omega(\lambda,\gamma)}=\frac{d_2[\int_\bR{J_2(y)e^{\lambda y}\cos(\gamma y)dy}-1] + r_2(b-1)(1-ab)}{\lambda}.\label{slg}
\eea
Note that $\varphi(\lambda,\gamma)$ is increasing in $\lambda$ for $\lambda>0$ and $\gamma\in(0,\pi/(2L))$, where $L$ is a positive constant such that $J_2(x)=0$ outside $[-L,L]$.
Moreover, $\inf_{\lambda>0}\omega(\lambda,\gamma)\to s_{**}$ as $\gamma\to 0$.

For a given $\ep\in (0,(\underline{s}^*-s)/4)$, we choose $\delta\in(0,1)$ such that $0<\delta<{(\underline{s}^*-s)/}{\ep}-4$ and $0<\ep<(\underline{s}^*-s)/(4+\delta)$.
Since $\inf_{\lambda>0}\omega(\lambda,\gamma)\to s_{**}$ as $\gamma\to0$, we get that
\be\label{order}
{s_{**}-\inf_{\lambda>0}\omega(\lambda,\gamma)\le \delta\ep}
\ee
for {all} $\gamma$ sufficiently small. 
On the other hand, from the definition of $s_{**}$ and $\lambda^*$, we have $\chi'(\lambda^*)=0$ and it follows that
\beaa
d_2\int_\bR{J_2(y)ye^{\lambda^* y}dy}=\chi(\lambda^*)=s_{**}.
\eeaa
Hence
\beaa
\varphi(\lambda^*,\gamma)=d_2\int_\bR{J_2(y)ye^{\lambda^* y}\frac{\sin(\gamma y)}{\gamma y}dy}\to d_2\int_\bR{J_2(y)ye^{\lambda^* y}dy}=s_{**}\;\mbox{ as $\gamma\to 0$.}
\eeaa

Now, we choose a small $\nu\in(0,\pi/(2L))$ such that $\varphi(\lambda^*,\nu)\ge s_{**}-\ep$ and \eqref{order} holds with $\gamma = \nu$.
Then, for this fixed~$\nu$, we claim that there exist $\lambda_1$ and $\lambda_2$ with $0<\lambda_1<\lambda_2<\lambda^*$ such that
\be\label{ld12}
{\varphi(\lambda_1,\nu)=s+\ep,\quad \varphi(\lambda_2,\nu)= \underline{s}^*-(1+\delta)\ep.}
\ee
Indeed, the existence of $\lambda_1$ and $\lambda_2$ follows from
\beaa
{0=\varphi(0,\nu)<s+\ep<\underline{s}^*-(1+\delta)\ep\le s_{**}-(1+\delta)\ep<s_{**}-\ep\le \varphi(\lambda^*,\nu)}
\eeaa
and the continuity of $\varphi$ in $\lambda$.

Let $\rho\in[\lambda_1,\lambda_2]$ be fixed and consider the function
$${\tilde{v}(x,t)=\eta\mu(x-\varphi t),}$$
where $\eta$ is a positive constant to be chosen later, the function $\mu$ (see \cite{w82}) is defined by
\begin{align*}
\mu(z):=\begin{cases}e^{-\rho z}\sin(\nu z), &\mbox{ for $0\le z\le \pi/\nu$},\\
0, &\mbox{ otherwise,}\end{cases}
\end{align*}
and, hereafter, $\varphi=\varphi(\rho,\nu)$ for convenience. We claim that there exist $\eta$ and $T$ (independent of $\rho$) such that
\be\label{claim1}
\tilde{v}_t\le d_2\mathcal{N}_2[\tilde{v}](x,t)+r_2\tilde{v}(x,t)[h(x,t)-\tilde{v}(x,t)]\ee
for any $t > T$ and $x \in \R$. We only need to consider the interval $\varphi t< x < \varphi t + \pi/\nu$ where the function $\tilde{v} (x,t)$ is not trivial.

To derive \eqref{claim1}, we first compute
$$\tilde{v}_t(x,t) = \eta\varphi e^{-\rho(x-\varphi t)}\left\{\rho \sin(\nu(x-\varphi t)) -\nu \cos(\nu(x-\varphi t))\right\}$$
for $t>0$ and $\varphi t< x < \varphi t + \pi/\nu$. Next, due to $\nu<\pi/(2L)$, we have for any $t>0$, $\phi t < x < \phi t + \pi / \nu$ and $y \in \operatorname{supp} (J_2) \subset [-L,L]$, that
$$\tilde{v}(x-y,t)\ge \eta e^{-\rho(x-y-\varphi t)}\sin(\nu(x-y-\varphi t)).$$
Then,
\beaa
&&d_2\left(\int_\bR{J_2(y)\tilde{v}(x-y,t)dy}-\tilde{v}(x,t)\right)+r_2\tilde{v}(x,t)[h(x,t)-\tilde{v}(x,t)]\\
&\ge& \eta d_2e^{-\rho(x-\varphi t)}\Big\{\int_\bR{J_2(y)e^{\rho y}\sin(\nu(x-y-\varphi t))dy}-\sin(\nu(x-\varphi t))\Big\}\\
&&\quad + r_2\eta e^{-\rho(x-\varphi t)}\sin(\nu(x-\varphi t))[h(x,t)-\eta\mu(x-\varphi t)]\\
&=&\eta d_2e^{-\rho(x-\varphi t)}\Big\{\int_\bR{J_2(y)e^{\rho y}\cos(\nu y)\sin(\nu(x-\varphi t))dy}\\
&&\qquad - \int_\bR{J_2(y)e^{\rho y}\sin(\nu y)\cos(\nu(x-\varphi t))dy}-\sin(\nu(x-\varphi t))\Big\}\\
&&\quad+ r_2\eta e^{-\rho(x-\varphi t)}\sin(\nu(x-\varphi t)[h(x,t)-\eta\mu(x-\varphi t)].
\eeaa
for $t>0$ and $\varphi t\le x \le \varphi t + \pi/\nu$. Hence, using~\eqref{vphi}, we find that \eqref{claim1} holds in this domain if and only if
\be\label{key2}
\rho \varphi\le d_2\left(\int_\bR{J_2(y)e^{\rho y}\cos(\nu y)dy}-1\right) + r_2[h(x,t)-\eta\mu(x-\varphi t)].
\ee
Since $\varphi(\lambda,\nu)$ is increasing in $\lambda$ and $\rho\in[\lambda_1,\lambda_2]$, by \eqref{ld12} we get
\beaa
s+\ep=\varphi(\lambda_1,\nu)\le \varphi = \varphi(\rho,\nu)\le\varphi(\lambda_2,\nu)=\underline{s}^{*}-(1+\delta)\ep<s^{**}-\ep .
\eeaa
Hence there exists $T_1>1$ sufficiently large (independent of $\rho \in [\lambda_1 , \lambda_2]$) such that
\be\label{key}
{[\varphi t, \varphi t +\pi/\nu]\subset [(s+\ep)t,(s^{**}-\ep)t]\;\mbox{ for $t\ge T_1$.}}
\ee
It follows from \eqref{ubar}, \eqref{key} and $h(x,t) = -1 + b \underline{u} (x,t)$ that
\beaa
\liminf_{t\to \infty} \inf_{x\in [\varphi t,\varphi t + \pi/\nu]} h(x,t)  \geq -1 + b (1-a (b-1)) = (b-1)(1-ab).
\eeaa
Then, for {a sufficiently small $\ep_1\in(0,\lambda_1 \ep)$}, there exists $T_2\ge T_1$ (also independent of $\rho$) such that
\beaa
h(x,t)\ge (b-1)(1-ab)-{\ep_1/r_2},
\eeaa
for any $t \geq T_2$ and $x\in [\varphi t,\varphi t + \pi/\nu]$.

Finally, we choose a sufficiently small positive constant $\eta$ such that
\beaa
r_2\eta<\lambda_1\ep-\ep_1.
\eeaa
Note that, {by \eqref{order},
\beaa
&&\omega(\rho,\nu)-\varphi(\rho,\nu)\ge \inf_{\lambda>0}\omega(\lambda,\nu)-\varphi(\lambda_2,\nu)\\
&=&\inf_{\lambda>0}\omega(\lambda,\nu)-[\underline{s}^{*}-(1+\delta)\ep] \ge \ep,\;\forall\,\rho\in[\lambda_1,\lambda_2].
\eeaa
Hence $\rho[\omega(\rho,\nu)-\varphi(\rho,\nu)]\ge\lambda_1[\omega(\rho,\nu)-\varphi(\rho,\nu)]\ge \lambda_1\ep$.}
Then, recalling also~\eqref{slg} we obtain
\beaa
\rho \varphi &\le& \rho\omega(\rho,\nu)-\lambda_1\ep \\
& < & \rho\omega(\rho,\nu)-\ep_1-r_2\eta\\
&=&d_2\left(\int_\bR{J_2(y) e^{\rho y}\cos(\nu y)dy}-1\right)+r_2(b-1)(1-ab)-\ep_1 -r_2\eta\\
&\le&{d_2\left(\int_\bR{J_2(y)e^{\rho y}\cos(\nu y)dy}-1\right) + r_2[h(x,t)-\eta\mu(x-\varphi t)]},
\eeaa
for $t \geq T_2$ and $x\in [\varphi t,\varphi t + \pi/\nu]$. Here we also used the fact that $\mu (z) \leq 1$ for any $z \geq 0$. Thus \eqref{key2} holds and so does~\eqref{claim1} for $t \ge T:= T_2$ and $x\in \R$.

Since $v_0$ is positive for some closed interval, we have that $\underline{v}(x,t)>0$ for all $x\in\bR,\,t>0$.
Up to reducing $\eta$, we can assume without loss of generality that
$$\underline{v} (x,T) \geq \eta \; \mbox{ for $x \in [(s+\varepsilon ) T ,  (\underline{s}^* - \varepsilon) T]$}.$$
In particular, we have
\beaa
\underline{v}(x,T)\geq \tilde{v}(x,T)\;\mbox{for $x\in [\varphi T,\varphi T + \pi/\nu]$,}
\eeaa
for any $\rho \in [\lambda_1,\lambda_2]$, due to~\eqref{ld12}. Also, $\underline{v}(x,T)>0=\tilde{v}(x,T)$ for $x < \varphi T$ and $x > \varphi T+\pi/\nu$.
Then, by comparison,
\be\label{order3}
v(x,t)\ge \underline{v}(x,t)\ge \tilde{v}(x,t)=\eta e^{-\rho(x-\varphi t)}\sin\left(\nu(x-\varphi t)\right),
\ee
for any $\rho \in [\lambda_1, \lambda_2]$, all $t \ge T$ and $x\in [\varphi t,\varphi t + \pi/\nu]$. On the other hand,
\be\label{goal}
\tilde{v}(x,t)\ge \frac{\eta}{\sqrt{2}} e^{-3\pi\rho/(4\nu)},\;\forall\, t\ge T , \, x\in[\varphi t+\pi/(4\nu),\varphi t+3\pi/(4\nu)] .
\ee
Since the estimate \eqref{goal} also holds for all $\rho\in[\lambda_1,\lambda_2]$, it follows from \eqref{order3} and another use of~\eqref{ld12} that
\be\label{goalb}
v(x,t)\ge \frac{\eta}{\sqrt{2}} e^{-3\pi\lambda_2/(4\nu)},\;\forall \, t \geq T, \, x\in[(s+\ep)t+\pi/(4\nu),(\underline{s}^*-(1+\delta)\ep)t+3\pi/(4\nu)].
\ee
Recalling that $\delta\in(0,1)$, we have that
\beaa
[(s+2\ep)t,(\underline{s}^*-2\ep)t]\subset[(s+\ep)t+\pi/(4\nu),(\underline{s}^*-(1+\delta)\ep)t+3\pi/(4\nu)]
\eeaa
for all $t\gg 1$. Then we conclude from~\eqref{goalb} that
\beaa
\liminf_{t\to\infty}\, \{\inf_{(s+2\ep)t\le x\le (\underline{s}^*-2\ep)t} v(x,t)\}\ge \frac{\eta}{\sqrt{2}} e^{-3\pi\lambda_2/(4\nu)} >0,
\eeaa
with $\ep$ arbitrarily small. This completes the proof of Theorem~\ref{vpos}.
\end{proof}



\section{The standard diffusion case}
\setcounter{equation}{0}

{In this section, we consider the diffusive predator-prey model with climate change, that is~\eqref{spp} which we recall here:}
\begin{equation*}
\begin{cases}
u_t(x,t)=d_{1} u_{xx}(x,t)+r_{1}u(x,t)[{\alpha(x-st)}-u(x,t)-av(x,t)], \; x\in\bR,\, t>0,\\
v_t(x,t)=d_{2} v_{xx}(x,t)+r_{2}v(x,t)[-1+bu(x,t)-v(x,t)], \; x\in\bR,\, t>0.
\end{cases}
\end{equation*}

First, the existence of forced waves for system \eqref{spp} can be derived {in the same fashion as for~\eqref{pp}.}
With a slight modification (cf. \cite{cgy17,hlr05,m01}), the proof is almost the same as that in \S3, and we safely omit it here.
{We also refer to~\cite[Theorem 1.1]{hz2017}) for the existence of forced waves in the scalar equation with standard diffusion,
from which we can infer lower solutions before applying a fixed point approach as in \S3.}

{Therefore, we will focus here on the issue of the spatio-temporal dynamics in the Cauchy problem \eqref{spp} and \eqref{pp-ic}.} Here we recall \eqref{sspeed}:
\beaa
s^*:=2\sqrt{d_1r_1},\; s_{*}:=2\sqrt{d_2r_2(b-1)}.
\eeaa
{{Actually, several results can again be proved in a very similar way as in the nonlocal case.
For the reader's convenience, we repeat all the statements and only point out the major differences of their proofs.}}

Similar to Theorem~\ref{sp1}, we first have

\begin{theorem}\label{ssp1}
{Assume that $(u_0,v_0) \in X_1 \times X_{b-1}$, and that both $u_0$ and $v_0$ have nonempty compact supports.} Then
\bea
&&\lim_{t\to\infty}u(x,t)=0 \mbox{ uniformly for $x\in\bR$, if $s>s^*$};\label{su-0}\\
&&\lim_{t\to\infty}v(x,t)=0 \mbox{ uniformly for $x\in\bR$, if $s>s_*$ .}\label{sv-0}
\eea
\end{theorem}
\begin{proof}
The proof is the same as that of Theorem~\ref{sp1}, by using Theorems 2.1 and 1.1 in \cite{lbsf2014}.
Note that we {do not need the perturbation with a $\delta$-term in this argument, by} applying \cite[Theorem 2.1]{lbsf2014}.
\end{proof}

Next, we have a similar result to Theorem~\ref{sp2} in the case {$s^* \leq s_*$. The proof is the same and therefore we omit it.}

\begin{theorem}\label{ssp2}
Suppose that {$s^*\leq s_*$}. Assume that {$(u_0,v_0) \in X_1 \times X_{b-1}$, and that} both $u_0$ and $v_0$ have nonempty compact supports.
If $s>s^*$, then
$$\lim_{t\to\infty}v(x,t)=0\, \mbox{ uniformly for $x\in\bR$.}$$
\end{theorem}


{Moreover, we have the following vanishing result as Theorem~\ref{sp4} in the outer cone regions.}

\begin{theorem}\label{ssp4}
{Assume that $(u_0,v_0) \in X_1 \times X_{b-1}$.
The following statements hold.}

{\rm (i)} 
{It holds}
\beaa
\lim_{t\to\infty}\sup_{x\le (s-\zeta)t}u(x,t)=0,\;\mbox{if {$s^*\geq s$},}\quad
\lim_{t\to\infty}\sup_{x\le (s-\zeta)t}v(x,t)=0,\;\mbox{if {$s_* \geq s$},}
\eeaa
for any small $\zeta>0$.

{\rm (ii)} If $u_0(x)=v_0(x)=0$ for $x\ge K$ for some constant $K$, then
$$\lim_{t\to\infty}\sup_{x\ge (s^*+\tau)t}u(x,t)=0,\;\;\lim_{t\to\infty}\sup_{x\ge (s_*+\tau)t}v(x,t)=0,  \; \;{\lim_{t\to\infty}\sup_{x\ge (s^*+\tau)t}v(x,t)=0},$$
for any $\tau>0$.
\end{theorem}
\begin{proof}
Here, instead of applying \cite[Theorem 3.3 (i)]{lwz18}, we apply \cite[Theorem 2.2 (i)]{lbsf2014} to prove part (i).
Notice that the condition $u_0(x)=0$ for all $-x\gg 1$ is not required in \cite[Theorem 2.2 (i)]{lbsf2014}. The proof of part (ii) is {almost the} same as before.
For the reader's convenience, we provide a detailed argument here.

Let $\tau>0$ be given and let $\lambda_1$ {the smaller} positive root of $$d_1\lambda^2 -(s^*+\tau/2)\lambda+ r_1=0.$$
Then, $\overline{u}(x,t):=A e^{-\lambda_1[x-(s^*+\tau/2)t]}$ is a solution of the following linear equation
$$\overline{u}_t(x,t)=d_1\overline{u}_{xx}(x,t) + r_1\overline{u}(x,t),$$
for any positive constant $A$.

Since $u_0(x)=0$ for $x\ge K$ and $u_0\le 1$, we can choose a positive constant $A$ large enough such that $u_0(x)\le A e^{-\lambda_1 x}$ for all $x\in\bR$.
Then, by the comparison principle,
$$0\le u(x,t)\le \overline{u}(x,t)=A e^{-\lambda_1[x-(s^*+\tau)t]} e^{-\lambda_1 \tau t/2}$$
for $x\in\bR$, $t\ge 0$.
In particular, we have $0\le u(x,t)\le A e^{-\lambda_1 \tau t/2}$ for $x\ge (s^*+\tau)t$ and $t\ge0$.
This implies that $\lim_{t\to\infty}\sup_{x\ge (s^*+\tau)t}u(x,t)=0$.

Next, we compare $v$ with the functions {$\overline{v}_1 := B_1 e^{-\lambda_{21} [x-(s_*+\tau/2)t]}$, where $\lambda_{21}$} is the smaller positive root of
$$d_2\lambda^2 -(s_*+\tau/2)\lambda+ r_2(b-1)=0,$$
and $\overline{v}_2 := \min \left\{ b-1 ,  B_2 e^{- \lambda_{22} [x - (s^* + \tau/2) t]} \right\}$,  where $\lambda_{22}$ is small enough so that
$$d_2 \lambda_{22}^2 - (s^* + \tau) \lambda_{22} - \frac{r_2}{2} < 0,$$
and $B_2$ is a large enough positive constant. {Then, proceeding as in the proof of Theorem~\ref{sp4}, we obtain the desired result for $v$ in (ii).}
The theorem is thereby proved.
\end{proof}

\begin{remark}
As in Remark~\ref{rk1}, for the standard diffusion case no species can survive in the long run without adjusting its habitat.
\end{remark}

Finally, similar to Theorem~\ref{sp3}, we have
\begin{theorem}\label{ssp3} {Let $(u_0,v_0)\in X_1\times X_{b-1}$. Assume that $u_0(x)>0$ on a closed interval and~$v_0$ has a nonempty compact support.
\begin{itemize}
\item If $s_* < s < s^*$, then
\beaa
\lim_{t\to\infty}\left\{\sup_{(s+\ep)t\le x\le (s^*-\ep)t}|u(x,t)-1|\right\}=0
\eeaa
for any $\ep\in(0,(s^*-s)/2)$.
\item If $s \leq s_* < s^*$, then
\beaa
\lim_{t\to\infty}\left\{\sup_{(s_*+\ep)t\le x\le (s^*-\ep)t}|u(x,t)-1|\right\}=0
\eeaa
for any $\ep\in(0,(s^*-s_*)/2)$.
\end{itemize}}
\end{theorem}
\begin{proof}
Here, instead of applying \cite[Theorem 3.3 (iii)]{lwz18}, we apply \cite[Theorem 2.2 (iii)]{lbsf2014} to prove the theorem.
\end{proof}
%

Now we turn our attention to the case when the changing speed is smaller than both $s^*$ and $s_*$, where we manage to improve our results compared with the nonlocal case. Hereafter we assume that $b>1$ and
$$s < \uns := \min \{ s_* , s^* \},$$
and we will prove Theorem~\ref{svpos}. In the next two subsections, we will deal consecutively with $u$ and $v$.
\subsection{{Survival of the prey}} Our proof is inspired by the methods in~\cite{dggs2021,dgm}. The first step is to investigate the supremum limit of solutions in the intermediate moving frames.
\begin{lemma}\label{lem1}
Assume that $s <  \uns$. Then for any $c \in (s, \uns)$ there exists $\delta_1 (c) >0$ such that, for any initial data satisfying $(u_0,v_0)\in X_1\times X_{b-1}$ with $u_0\not\equiv 0$ and $v_0 \not \equiv 0$,
the corresponding solution $(u,v)$ of \eqref{spp} satisfies
$$\limsup_{t \to +\infty} {u}(ct,t)  \geq \delta_1 (c).$$
\end{lemma}
Let us point out that the same result holds (with $s= - \uns$) for solutions of the homogeneous predator-prey model
\begin{equation}\label{eq:alpha=1}
\begin{cases}
u_{t} =d_{1} u_{xx}+r_{1}u (1-u-av),\\
v_{t} =d_{2} v_{xx} +r_{2}v (-1+bu-v),
\end{cases}
\end{equation}
which is the same as~\eqref{spp} with instead $\alpha \equiv 1$. Since such a result is also needed for our argument, we state it below and refer to~\cite{dgm} for the details. As a matter of fact, our proof of Lemma~\ref{lem1} below roughly follows the same argument.
\begin{lemma}[\cite{dgm}; Lemma~5.2]\label{lem0}
For any $c \in [0, \uns)$ there exists $\delta'_1 (c) >0$ such that, for any initial data satisfying $(u_0,v_0)\in X_1\times X_{b-1}$ with $u_0\not\equiv 0$ and $v_0 \not \equiv 0$,
the corresponding solution $(u,v)$ of \eqref{eq:alpha=1} satisfies
$$\limsup_{t \to +\infty} {u}(ct,t)  \geq \delta'_1 (c).$$
\end{lemma}
\begin{proof}[Proof of Lemma~\ref{lem1}]
We assume by contradiction that there exists a sequence of initial data $(u_{0,n}, v_{0,n})$ in $X_1 \times X_{b-1}$ such that, for all $n$,
$$u_{0,n} \not \equiv 0, \quad v_{0,n} \not \equiv 0,$$
and
$$\lim_{n \to +\infty} \limsup_{t \to +\infty} u_n (ct,t) = 0.$$
Here $(u_n,v_n)$ naturally denotes the solution of \eqref{spp} associated with $(u_{0,n}, v_{0,n})$. Then we can choose a time sequence $t_n \to +\infty$ as $n\to\infty$ such that
$$\lim_{n \to +\infty} \sup_{ t \geq t_n} u_n ( ct, t) = 0.$$
Then we claim that also, for any $R >0$,
\begin{equation}\label{new_limun_pos}
\lim_{n \to +\infty} \sup_{ |x| \leq R, t \geq t_n} u_n (x + ct, t) =0.
\end{equation}
Indeed, assume by contradiction that there exist sequences $x_n \in [-R,R]$ and $t'_n \geq t_n$ such that
$$
\liminf_{n \to +\infty} u_n (x_n + c t'_n, t'_n) > 0.
$$
By standard parabolic estimates, we may assume up to extraction of a subsequence that
$$(u_n ,v_n) (x + ct'_n  , t+ t'_n) \to (u_\infty ,v_\infty) (x,t),$$
where the convergence is understood in the locally uniform sense, and $(u_\infty, v_\infty)$ is an entire in time solution of system \eqref{eq:alpha=1},
due to $c > s$. Moreover, by construction we have on the one hand that $u_\infty (0,0) = 0$, hence
$$u_\infty \equiv 0,$$
by the strong maximum principle. Yet on the other hand, we also have up to extraction of another subsequence that $x_n \to x_\infty \in [-R,R]$ and then that $u_\infty (x_\infty, 0) >0$, a contradiction. We conclude that \eqref{new_limun_pos} holds.

In the same way we can prove that
\begin{equation}\label{new_limvn_pos}
\lim_{n \to +\infty} \sup_{ |x| \leq R, t \geq t_n} v_n (x + ct, t) =0.
\end{equation}
Indeed, if not then we find an entire in time solution $(u_\infty,v_\infty)$ of~\eqref{eq:alpha=1} with $u_\infty \equiv 0$ and $v_\infty > 0$. Then
$$(v_{\infty})_{t}  = d_{2} (v_{\infty})_{xx}  + r_2 v_\infty {(-1 - v_\infty ),}$$
for all $t \in \R$ and $x \in \R$. Since $v_\infty$ is also uniformly bounded from above by $b-1$ by construction, one may infer that $v_\infty \equiv 0$, another contradiction. The claim~\eqref{new_limvn_pos} is now proved.

Now define $\lambda_R := - \frac{\pi^2}{4R^2} $ and $\varphi_R (x) := \cos \left( \frac{\pi}{2R} x \right)$ the solution of the principal eigenvalue problem
\begin{equation*}
\begin{cases}
d_1 ( \varphi_{R})_{xx} = \lambda_R \varphi_R, \  x \in (-R,  R), \\
\varphi_{R} \left( \pm R \right) = 0, \quad \varphi_R >0 \ \mbox{ in } (-R,R).
\end{cases}
\end{equation*}
Hereafter we extend $\varphi_R$ by $0$ outside the interval $[-R,R]$. Next, thanks to $c < s_*$, we can choose $\delta>0$ small enough so that
$$ \frac{c^2}{4d_1} < r_1 (1- 2 \delta) ,$$
and $R$ large enough so that
$$- \lambda_R < r_1 (1-2 \delta) - \frac{c^2}{4d_1} .$$
Then, by \eqref{new_limun_pos} and \eqref{new_limvn_pos}, we have that there exists $n$ large enough such that
$$(u_{n})_{t} (x,t) \geq d_1 (u_{n})_{xx} (x,t) + r_1 (1-\delta) u_n (x,t),$$
for all $t \geq t_n$ and $x \in (ct - R, ct +R)$. On the other hand, by a straightforward computation the function
$$\underline{u} (x,t) := A e^{r_1 \delta t} e^{-\frac{c}{2 d_1}( x-ct)} \varphi_R (x-ct)$$
satisfies, for any choice of $A>0$,
$$\underline{u} (x,t) \leq d_1 \underline{u} (x,t) + r_1 (1-\delta)\underline{u} (x,t),$$
in the whole domain $\R \times \R$. Taking $A$ small enough so that $u_n (t=t_n) \geq \underline{u} (t=t_n)$, we get by the comparison principle that
$$u_n (x,t) \geq \underline{u} (x,t)$$
for all $t \geq t_n$ and $x \in \R$. However $\underline{u} (ct,t) \to +\infty$ as $t \to +\infty$, which contradicts the boundedness of $u_n$. Finally we conclude that Lemma~\ref{lem1} holds true.
%
%
\end{proof}
Next, we improve the previous lemma by showing that the infimum limits are also positive.
\begin{lemma}\label{lem2}
Assume that $s <  \uns$. Then for any $c \in (s, \uns)$ there exists $\delta_2 (c) >0$ such that, for any initial data satisfying $(u_0,v_0)\in X_1\times X_{b-1}$ with $u_0\not\equiv 0$ and $v_0 \not \equiv 0$,
the corresponding solution $(u,v)$ of \eqref{spp} satisfies
$$\liminf_{t \to +\infty} {u}(ct,t)  \geq \delta_2 (c).$$
\end{lemma}
\begin{proof}
We again proceed by contradiction, and assume that there is a sequence of initial data $(u_{0,n}, v_{0,n}) \in X_1 \times X_{b-1}$ with $u_{0,n} \not \equiv 0 \not \equiv v_{0,n}$, and a time sequence $t_n \to +\infty$ such that the associated solution $(u_n,v_n)$ satisfies
$$\lim_{n \to +\infty} u_n (ct_n, t_n) = 0.$$
According to Lemma~\ref{lem1}, we know that there also exists another sequence $t'_n \to +\infty$ such that
$$u_n (ct'_n,t'_n) \geq \frac{\delta_1 (c)}{2} ,$$
and without loss of generality we can choose it so that $t'_n < t_n$ for any $n$. Now define
$$\tau_n := \sup \left\{ t  \geq t'_n \mid  u_n (c t,t) \geq \frac{\delta_1 (c)}{2} \right\}.$$
It immediately follows that
$$\forall t \in (\tau_n , t_n), \quad u_n (ct, t) \leq \frac{\delta_1 (c)}{2}.$$
Moreover, by standard parabolic estimates, one can extract a subsequence so that $(u_n,v_n) (x+ ct_n , t+  t_n)$ converges locally uniformly to an entire in time solution of \eqref{eq:alpha=1}. By construction and the strong maximum principle, its $u$-component must be identically equal to 0. In particular, if $t_n - \tau_n$ admits a bounded subsequence, then
$$u_n (c \tau_n, \tau_n) \to 0,$$
as $n \to +\infty$, which contradicts the fact that
$$u_n (c \tau_n, t_n) = \frac{\delta_1 (c)}{2},$$
for $n$ large enough. Thus we obtain that
$$t_n - \tau_n \to +\infty,$$
as $n \to +\infty$.

This allows us to extract a subsequence so that
$$(u_n,v_n) ( \cdot + c \tau_n , \cdot + \tau_n) \to (u_\infty, v_\infty)$$
where the convergence is understood in the locally uniform sense, the pair $(u_\infty, v_\infty)$ is an entire in time solution of \eqref{eq:alpha=1}, and by construction
$$u_\infty (0,0) = \frac{\delta_1 (c)}{2},$$
\begin{equation}\label{u_impossibly_small}
u_\infty (ct,t) \leq \frac{\delta_1 (c)}{2},
\end{equation}
for all $t \geq 0$.

Notice that, when $v_\infty \not \equiv 0$, this entire solution immediately contradicts Lemma~\ref{lem0} above. Now it remains to consider the case when $v_\infty \equiv 0$. Then
$$(u_{\infty})_t =d_{1} (u_{\infty})_{xx} +r_{1}u_\infty  [1-u_\infty],$$
which is the standard scalar reaction-diffusion equation of the KPP type. Recall also that $c < s^*$. Then, by a result of Aronson and Weinberger~\cite{aw75}, due to $u_\infty (t= 0 ) \geq \not \equiv 0$ and the so-called hair-trigger effect, it is known that
$$u_\infty (ct,t) \to 1,$$
as $t \to +\infty$, which again contradicts~\eqref{u_impossibly_small}.
{The lemma is thereby proved.}
\end{proof}
We can finally turn to some uniform lower estimate for the prey in the intermediate range between the moving frames with speeds $s$ and $\uns = \min \{ s_* , s^*\}$.
\begin{theorem}\label{spprey}
Assume that $(u_0,v_0) \in X_1 \times X_{b-1}$ and that both $u_0$ and $v_0$ have nonempty compact supports.  {If $s < \uns$,} then for any $\eta\in(0,(\uns -s)/2)$ there is a positive constant
{$\kappa_1$} (independent of $(u_0,v_0)$) such that
\be\label{sprey-lower}
\liminf_{t \to +\infty} \left\{\inf_{(s+\eta)t \leq x\leq ( \uns - \eta) t} u(x,t) \right\}\ge\kappa_1.
\ee
\end{theorem}
\begin{proof}
We fix $\eta$ and use another argument by contradiction, which is quite similar to the one in the proof of Lemma~\ref{lem2}. If the conclusion of Theorem~\ref{spprey} does not hold true, then one can find a sequence of initial data $(u_{0,n},v_{0,n}) \in X_1 \times X_{b-1}$ with $u_{0,n} \not \equiv 0 \not \equiv v_{0,n}$, as well as sequences {$\{t_{n,k}\}$ and $\{x_{n,k}\}$} with $t_{n,k} \to +\infty$ as $k \to +\infty$, and
$$x_{n,k} \in  \left[ (s+\eta) t_{n,k} , (\uns - \eta) t_{n,k} \right],$$
such that
$$ u_n (x_{n,k} , t_{n,k}) \leq \frac{1}{n},$$
for any {positive} integers $n$ and $k$. However, according to Lemma~\ref{lem2}, we have that
$$\liminf_{t \to +\infty} u_n \left( \left(\uns - {\eta}/{2} \right) t, t \right) \geq \delta_2 \left( \uns - {\eta}/{2} \right).$$
In particular, we can define another time sequence
$$t'_{n,k} := \frac{x_{n,k}}{\uns - {\eta}/{2}},$$
which is such that
$$t'_{n,k} < t_{n,k}, \quad \lim_{k \to +\infty} t'_{n,k} =  +\infty,$$
and
$$u_n ( x_{n,k} , t'_{n,k}) \geq \frac{\delta_2  \left( \uns - {\eta}/{2} \right)}{2}$$
for any $k$ large enough. For each $n$, we choose such a large $k$ and drop it from our notation for convenience. Then we can define
$$\tau_n := \sup \left\{ t  \geq t'_n \mid  u_n (c t,t) \geq \frac{\min \{ \delta'_1 (c), \delta_2 \left( \uns - {\eta}/{2} \right) \}}{2} \right\},$$
where $\delta'_1$ comes from Lemma~\ref{lem0}. As in the proof of Lemma~\ref{lem2}, we have that
$$t_n - \tau_n \to +\infty,$$
as $n \to +\infty$, and up to extraction of a subsequence
$$(u_n,v_n) (\cdot + c \tau_n , \cdot + \tau_n) \to (u_\infty, v_\infty)$$
in the locally uniform sense, where $(u_\infty, v_\infty)$ is an entire in time solution of \eqref{eq:alpha=1} such that
$$u_\infty (0,0) = \frac{\min \{ \delta'_1 (c), \delta_2 \left( \uns - {\eta}/{2} \right) \}}{2},\;
u_\infty (0,t) \leq \frac{\min \{ \delta'_1 (c), \delta_2 \left( \uns - \frac{\eta}{2} \right) \}}{2}\;\mbox{ for all $t \geq 0$.}$$
Regardless of whether $v_\infty \equiv 0$ or $v_\infty >0$, we reach a contradiction with either the result of Aronson and Weinberger~\cite{aw75}, or Lemma~\ref{lem0} from~\cite{dgm}. This concludes the proof.
\end{proof}

\subsection{Survival of the predator} We now turn to the persistence of the predator $v$ in the moving frames with speeds in the interval $(s, \uns)$, where we recall that $\uns = \min \{ s_*, s^*\}$.
\begin{theorem}\label{sppred}
Assume that $(u_0, v_0) \in X_1 \times X_{b-1}$ and that both $u_0$ and $v_0$ have nonempty compact supports. If {$s < \uns$,} then for any $\eta\in(0,(\uns-s)/2)$ there is a positive constant {$\kappa_2$}
(independent of $(u_0,v_0)$) such that
$$\liminf_{t \to +\infty} \{\inf_{(s+\eta)t \leq x\leq ( \uns - \eta) t} v(x,t)\}{\ge \kappa_2.}$$
\end{theorem}
The method is the same as for the prey, though it is made slightly easier {by} the fact that we already know that $u$ can never approach 0 in these moving frames as $t \to +\infty$. In particular, before proving Theorem~\ref{sppred}, we prepare two lemmas.
\begin{lemma}\label{lem1_v}
Assume that $\uns>s$. Then for any $c \in (s, \uns)$ there exists $\delta_3 (c) >0$ such that, for any initial data satisfying $(u_0,v_0)\in X_1\times X_{b-1}$ with $u_0\not\equiv 0$ and $v_0 \not \equiv 0$,
the corresponding solution $(u,v)$ of \eqref{spp} satisfies
$$\limsup_{t \to +\infty} {v}(ct,t)  \geq \delta_3 (c).$$
\end{lemma}

\begin{proof}
We let $c\in(s,\uns)$, and assume by contradiction that there exists a sequence of solutions $\{(\tilde{u}_n , \tilde{v}_n)\}$ with initial data $\{(u_{0,n},v_{0,n})\}\subset X_1\times X_{b-1}$, with
$u_{0,n}\not\equiv 0$ and $v_{0,n} \not \equiv 0$, such that
$$\lim_{n \to +\infty} \limsup_{t \to +\infty}  \tilde{v}_n (ct,t)  = 0.$$
Then, for each $n$, we can choose $t_n$ large enough such that
\be\label{eqvan}
\lim_{n \to +\infty} \sup_{ t\geq t_n} \tv_n(ct,t) = 0.
\ee
By passing to the limit as $n\to\infty$, and applying the strong maximum principle, we also have
\be\label{v-small}
\limsup_{n\to\infty}\{\sup_{t\ge t_n,|x-ct|\le R}\tv_n(x,t)\}=0
\ee
for any $R>0$.

Next, we claim that
\be\label{claimR}
\limsup_{n\to\infty}\{\sup_{t\ge t_n,|x-ct|\le R}\tu_n(x,t)\}=1,
\ee
for any $R>0$.
For contradiction, we assume that there is a sequence $\{(x_n,t'_n)\}$ with $t'_n\ge t_n$ and $x_n\in[ct'_n-R,ct'_n+R]$ such that $\limsup_{n\to\infty}\tu_n(x_n,t'_n)<1$.
Then, by standard parabolic estimates and extracting a subsequence, we have that $(\tu_n,\tv_n)(x+x_n,t+t'_n)$ converges to $(u_\infty,v_\infty)$ as $n\to\infty$ for some entire solution $(u_\infty,v_\infty)$ of \eqref{eq:alpha=1}.

Since $v_\infty(0,t)=0$ for all $t>0$, by the strong maximum principle we get that $v_\infty\equiv 0$. In particular, $u_\infty$ satisfies
\beaa
(u_\infty)_t=d_1(u_\infty)_{xx}+r_1u_\infty(1-u_\infty),\; (x,t)\in\bR^2.
\eeaa
On the other hand, by Theorem~\ref{spprey}, we have that
$$\inf_{ (x,t) \in \R^2 } u_\infty(x,t) >0.$$
This implies that $u_\infty\equiv 1$, a contradiction to $u_\infty(0,0)<1$ by our choices of $x_n$ and $t'_n$. Hence~\eqref{claimR} holds.

Now, for any small $\eta>0$ and large $R>0$, we have
\beaa
(\tv_n)_t\ge d_2(\tv_n)_{xx}+r_2(b-1-2\eta)\tv_n,\; |x-ct_n|\le R,\, t\ge t_n,
\eeaa
for any $n$ large enough. As in the proof of Lemma~\ref{lem1}, we infer that $\tilde{v}_n (ct,t) \to +\infty$ as $t \to +\infty$, a contradiction to~\eqref{v-small}. The lemma is thereby proved.
\end{proof}
\begin{lemma}\label{lem2_v}
Assume that $\uns>s$.
Then for any $c \in (s, \uns)$ there exists $\delta_4 (c) >0$ such that, for any initial data satisfying $(u_0,v_0)\in X_1\times X_{b-1}$ with $u_0\not\equiv 0$ and $v_0 \not \equiv 0$, the corresponding solution $(u,v)$ of \eqref{spp} satisfies
$$\liminf_{t \to +\infty} v(ct,t)  \geq \delta_4 (c).$$
\end{lemma}
\begin{proof}
The method is the same as that of Lemma~\ref{lem2}. Fixing $c \in (s, \uns)$ and proceeding by contradiction, we find an entire solution $(u_\infty, v_\infty)$ of \eqref{eq:alpha=1}, such that also
\begin{equation}\label{almost_vinfty}
v_\infty (0,0) = \frac{\delta_3 (c)}{2}, \quad v_\infty (ct,t) \leq \frac{\delta_3 (c)}{2},
\end{equation}
for all $t \geq 0$. Moreover, by Theorem~\ref{spprey} we have that
$$\inf_{ (x,t) \in \R^2} u_\infty (x,t) >0,$$
which in turn implies that $u_\infty \equiv 1$. Thus
$$(v_\infty)_t = d_2 (v_\infty)_{xx} + r_2 v_\infty (b-1 - v_\infty).$$
Since $c < s^*$, by the classical result of Aronson and Weinberger~\cite{aw75}, we have that
$$v_\infty (ct,t) \to b-1,$$
as $t \to +\infty$, a contradiction with \eqref{almost_vinfty}.
\end{proof}
Theorem~\ref{sppred} now follows from the previous lemmas. Since the proof is similar to that of Theorem~\ref{spprey}, we omit the details.

\subsection{Proof of Theorem~\ref{svpos}.}

Now, we can finish the proof of Theorem~\ref{svpos}. Indeed, take any sequence $t_n \to +\infty$ and $x_n \in [(s+\eta) t, (\uns - \eta) t]$ such that $(u,v) (\cdot + x_n, \cdot + t_n)$ converges as $n \to +\infty$.
Then, putting together Theorems~\ref{spprey} and~\ref{sppred}, we get as $n \to +\infty$ an entire in time solution $(u_\infty,v_\infty)$ of \eqref{eq:alpha=1}, such that
$$\kappa_1 \leq u_\infty (x,t) \leq 1,\quad \kappa_2 \leq v_\infty (x,t) \leq b-1,$$
for all $(x,t) \in \R^2$. However, according to \cite[Lemma 4.1]{dggs2021}, it follows from a Lyapunov argument that any such entire in time solution must coincide with
$$ (u_*,v_*) = \left( \frac{1+a}{1+ab} , \frac{b-1}{1+ab} \right).$$
By standard parabolic estimates that ensure the sequential compactness of the set of time and space shifts of the solution, one may finally infer that Theorem~\ref{svpos} holds true.


\end{document}